\documentclass[12pt]{amsart}
\usepackage[latin1]{inputenc}
\usepackage{indentfirst}
\usepackage[active]{srcltx}
\usepackage{amsfonts}
\usepackage{graphicx}
\usepackage{amsmath}
\usepackage{amssymb}
\usepackage{latexsym}
\usepackage{amscd}
\usepackage[all]{xy}
\usepackage{color}
\usepackage{tikz}
\usepackage{breqn}
\usepackage{subfig}
\usetikzlibrary{shapes.geometric}
\usepackage{algorithm}
\usepackage{algorithmic}
\usepackage{hyperref}

\def\negh{\mathbf h}

\def\negs{\mathbf s}

\def\negx{\mathbf x}
\def\negy{\mathbf y}
\def\negz{\mathbf z}

\def\negc{\mathbf c}

\def\negf{\mathbf f}
\def\negg{\mathbf g}

\def\neg1{\text{\boldmath$1$}}

\def\neg1{\text{\boldmath$1$}}

\DeclareMathOperator{\lub}{lub}

\newtheorem{theorem}{Theorem}[section]
\newtheorem{definition}[theorem]{Definition}
\newtheorem{lemma}[theorem]{Lemma}
\newtheorem{corollary}[theorem]{Corollary}
\newtheorem{proposition}[theorem]{Proposition}
\newtheorem{remark}[theorem]{Remark}
\newtheorem{example}[theorem]{Example}

\title[On Atoms of the set of Generalized Numerical Semigroups with fixed corner element]{On Atoms of the set of Generalized Numerical Semigroups with fixed corner element}

\author[M. Bernardini]{Matheus Bernardini}
\address{Universidade de Bras\'{i}lia, Faculdade do Gama, Bras\'{i}lia, DF, Brazil}
\email{matheusbernardini@unb.br}

\author[A. Castellanos]{Alonso S. Castellanos}
\address{Universidade Federal de Uberl\^{a}ndia, Faculdade de Matem\'{a}tica, Uberl\^andia, MG, Brazil}
\email{alonso.catellanos@ufu.br}

\author[W. Ten\'orio]{Wanderson Ten\'orio}
\address{Universidade Federal do Mato Grosso, Instituto de Ci\^{e}ncias Exatas e da Terra, Cuiab\'{a}, MT, Brazil} 
\email{wanderson.tenorio@ufmt.br}

\author[G. Tizziotti]{Guilherme Tizziotti}
\address{Universidade Federal de Uberl\^{a}ndia, Faculdade de Matem\'{a}tica, Uberl\^andia, MG, Brazil}
\email{guilhermect@ufu.br}

\thanks{{\em 2020 Math. Subj. Class.}: Primary 20M14; Secondary 06F05, 11D07}

\thanks{{\em Keywords}: generalized numerical semigroup, atoms, corner element}

\thanks{The author Alonso S. Castellanos was partially funded by FAPEMIG APQ 00696-18 and RED-0013-21. The author Guilherme Tizziotti was partially funded by CNPq 307037/2019-3.}

\begin{document}

	\maketitle
	
	\begin{abstract}
       We study the so-called atomic GNS, which naturally extends the concept of atomic numerical semigroup. We introduce the notion of corner special gap and we characterize the class of atomic GNS in terms of the cardinality of the set of corner special gaps and also in terms of a maximal property. Using this maximal property we present some properties concerning irreducibility of Frobenius GNSs. In particular, we provide sufficient conditions for certain Frobenius GNSs to be an atom non-irreducible (ANI). Furthermore, we given necessary and sufficient conditions so that the maximal elements of a set of Frobenius GNSs with two fixes gaps to be all irreducible or not. 

%Given an element $\mathbf{c}\in \mathbb{N}_{0}^d$, we denote by $\mathcal{F}(\mathbf{c})$ the set of all generalized numerical semigroups in $\mathbb{N}_{0}^d$ of corner $\mathbf{c}$. The set $(\mathcal{F}(\mathbf{c}), \cap)$ is a semigroup. In this paper we study the generators of this semigroup.  
	\end{abstract}

	\section{Introduction}
Let $\mathbb{N}_0$  be the set of non-negative integers and $d$ be a positive integer. A \textit{generalized numerical semigroup} (GNS) is a subset of $\mathbb{N}^d_0$ such that it is closed under the usual addition in $\mathbb{N}_0^d$, it contains the zero element of $\mathbb{N}^d_0$ and its complement is finite. If $S$ is a GNS, then the elements of $\operatorname{H}(S)=\mathbb{N}_0^d\setminus S$ are called the \textit{gaps} of $S$ and the genus of $S$ is the cardinality of the set of gaps of $S$. The concept of GNS was first introduced in \cite{FPU}, as a natural generalization of numerical semigroups, which is the case when $d=1$. For a contextualization on numerical semigroups and the state of art of some open problems, we indicate \cite{GS-R,Kaplan}. 

Recently, some papers have been published concerning the study of the GNSs, mostly inspired by definitions on the numerical semigroup case. For instance, the problem of counting GNSs by genus have been considered by Failla, Peterson, and Utano \cite{FPU} and by Cisto, Delgado, and Garc­\'ia-S\'anchez \cite{CDG}. Generalizations of the Wilf's conjecture have been considered by Cisto \textit{et al.} \cite{CDFFPU}. Counting Frobenius GNSs with a fixed Frobenius element have been studied by Singhal and Lin \cite{SL}. Cisto, Failla, Peterson, and Utano \cite{CFPU} studied the irreducible GNSs and Cisto, and Ten\'orio \cite{CW} introduced the concept of almost symmetry for GNSs. We recall that a GNS is {\it irreducible} if it cannot be expressed as the intersection of two GNSs containing it properly. In 2009, Rosales \cite{Rosales} introduced the so-called {\em atomic numerical semigroup}; he proved several properties of this class of numerical semigroups and, in particular, he characterized the class of atomic and non-irreducible numerical semigroup as the one that have exactly two special gaps. Atomic numerical semigroups have also been studied in \cite{Robles-Rosales}. %Observe that the definition of atomic numerical semigroup can be done by changing Frobenius number by conductor.

%\textcolor{red}{
%In \cite{CFPU}, the authors showed that all generalized numerical semigroups can be expressed as an intersection of finitely many irreducible generalized numerical semigroups. Moreover, they proved that, for an irreducible GNS the Frobenius element is unique with respect to any relaxed monomial order.}

In this work, we are interested in dealing with the generalization of atomic numerical semigroups to higher dimensions. Since the set of Frobenius elements of a GNS can have more than one element, we consider the class of irreducible GNSs that preserve its corner element under a unitary extension, which is the class of atomic GNSs. We obtain a characterization of the class atomic GNSs as the one that have at most one corner special gap (see Definition \ref{c special gaps}), which is an important tool to deal with atomic GNSs. Moreover, we relate the class of atomic GNSs with the class of GNSs that have a maximal property.

Here is an outline of the paper. In Section 2, we present some basic results and terminologies concerning generalized numerical semigroups, that are useful throughout the paper. In Section 3, we deal with the so-called corner special gaps of a GNS, which are the special gaps that preserve the corner element by unitary extensions. In Section 4, we introduce the concept of atoms, which have been studied for numerical semigroups; this concept is characterized by using the cardinality of the set of corner special gaps (see Theorem~\ref{prop atomo c especial}) and also by a maximal property (see Theorem~\ref{atom MF}). Using the notion of this maximal property, in Section 5, we present some properties concerning irreducibility of Frobenius GNSs. In particular, we provide sufficient conditions for certain Frobenius GNSs to be an atom non-irreducible (ANI) (see Proposition \ref{prop ANI}). Furthermore, we given necessary and sufficient conditions so that the maximal elements of a set of Frobenius GNSs with two fixes gaps to be all irreducible or not (see Theorem \ref{teo ANI} and Corollary \ref{Cor GNS Irreducible}).  
%\textcolor{red}{
%In Section 5, the main result is Theorem \ref{TmainS5}, where we describe those sets $\mathbf{g_2}, \mathbf{g_1}\subset \mathbb{N}^d$ such that $ML(g_1,g_2)$ contain at lest a non-irreducible generalized numerical semigroup.}

%
%
%---------------
%	Itens a serem verificados:
%	
%	\begin{itemize}
%	  \item $\mathbb{N}$ e $\mathbb{N}_0$ ({\color{blue} OK - Alonso})
%	  \item $EH(S)$ e $\operatorname{EH}(S)$ assim como $H(S)$, $PF(S)$ e $CEH(S)$ ({\color{blue} OK - Alonso})
%	  \item Ã­ndices (em cima/embaixo) - padronizar (negrito/sem negrito)
%	  \item NotaÃ§Ãµes no geral
%	\end{itemize}
%	
%	...
	
\section{Preliminaries}

In this section we collect some basic tools and results that will be used throughout this text. We start by recalling the notion of corner element of GNS introduced in \cite{BTT}.
 
\begin{definition}
Let $S \subseteq \mathbb{N}_0^d$ be a GNS. An element $\negc = (c_1, \ldots, c_d) \in S$ is called a \emph{corner element} of $S$ if the following conditions are satisfied:
\begin{itemize}
  \item[(1)] for all $i \in \{1,\ldots,d\}$ and for all $x \in \mathbb{N}_0$ such  that $x \geq c_i$ we have that \newline $(y_1, \ldots, y_{i-1}, x , y_{i+1}, \ldots , y_d) \in S$ for all $y_1, \ldots, y_{i-1}, y_{i+1}, \ldots , y_d \in \mathbb{N}_0$;
  \item[(2)] for all $i \in \{1,\ldots,d\}$, we have that $(z_1, \ldots, z_{i-1}, c_{i} - 1, z_{i+1}, \ldots , z_d) \notin S$ for some $z_1, \ldots, z_{i-1}, z_{i+1}, \ldots , z_d \in \mathbb{N}_0$.
  \end{itemize}
\end{definition}

The corner element of a GNS $S$ is unique (see \cite[Proposition 3.2]{BTT}) and will be denoted by $\mathbf{c}(S)$. Concerning the coordinates of the corner element of a GNS $S$ in $\mathbb{N}_0^d$, it is known from \cite[Proposition 3.3]{BTT} that if $S$ has positive genus and $\mathbf{c}(S) =(c_1,\ldots,c_d)$, then $c_i \neq 0$ for all $i \in \{1,\ldots,d\}$, and there is $i \in \{1,\ldots,d\}$ such that $c_i>1$. 

Recall that the \emph{least upper bound} ($\lub$) of a finite set $\mathcal{B}\subseteq \mathbb{N}_0^d$ is the element $$\mbox{lub}(\mathcal{B}):=(\max\{x_1\mbox{ : } \negx \in \mathcal{B} \},\ldots,\max\{x_d\mbox{ : } \negx \in \mathcal{B} \})\in \mathbb{N}_0^d.$$

Hence, the corner element can be related with the gaps of a GNS as follows.

%\begin{proposition} \cite[Proposition 4.5]{CFPU} \label{lub}
\begin{proposition} \cite[Theorem 3.5]{BTT} \label{lub}
    Let $S\subseteq \mathbb{N}_0^d$ be a GNS with positive genus and corner element $\negc$. Then $$\negc=\operatorname{lub}(\operatorname{H}(S))+\mathbf{1}.$$
\end{proposition}

The natural partial order $\leq$ in $\mathbb{N}_0^d$ stands for the relation defined as: for $\negx, \negy \in \mathbb{N}_0^d$, 
$$\negx\leq \negy \ \ \mbox{if and only if} \ \ x_i\leq y_i \ \ \mbox{for all } i\in \{1,\ldots, d\}.$$
If $\negx\leq \negy$ and $\negx \neq \negy$, we just write $\negx< \negy$.

When there is a unique maximal element in $\operatorname{H}(S)$ with respect to the natural partial order $\leq$ of $\mathbb{N}_0^d$, $S$ is said to be a \emph{Frobenius GNS}. In this case, the maximal element is called the \emph{Frobenius element} of $S$ and denoted by $\operatorname{F}(S)$. In particular, Proposition~\ref{lub} yields that $S$ is a Frobenius GNS if and only if $\negc(S)-\mathbf{1}\in \operatorname{H}(S)$, and in this case $\operatorname{F}(S)=\negc(S)-\mathbf{1}$.

\begin{lemma} \label{corner subset} Let $S$ and $T$ be two GNSs in $\mathbb{N}_0^d$. If $S \subseteq T$, then $\mathbf{c}(T) \leq \mathbf{c}(S)$.
\end{lemma}
\begin{proof}
    If $S \subseteq T$, then as $\operatorname{H}(T) \subseteq \operatorname{H}(S)$ we have $\operatorname{lub}(\operatorname{H}(T))\leq \operatorname{lub}(\operatorname{H}(S))$, and the conclusion thus follows from Proposition~\ref{lub}. 
\end{proof}

As a consequence, Lemma~\ref{corner subset} yields the following.

\begin{corollary} \label{lemma corner subset}
		Let $S_1, S_2 \subseteq \mathbb{N}_{0}^{d}$ be GNSs  with corner element $\mathbf{c}$. If $S$ is a GNS such that $S_1 \subseteq S \subseteq S_2$, then $S$ has also corner element $\mathbf{c}$.
\end{corollary}

For $S^*=S\setminus \{\mathbf{0}\}$, the set of \emph{pseudo-Frobenius elements} of a GNS $S$ is defined by $$\operatorname{PF}(S) := \{\negx \in \operatorname{H}(S) \ : \ \negx + S^* \subseteq S\}.$$
The elements in the subset
$$\operatorname{EH}(S):= \{ \mathbf{x} \in \operatorname{PF}(S) \ : \ 2 \mathbf{x} \in S \}$$
are called \emph{special gaps} of $S$ and play an important role in the theory because of the following property regarding unitary extensions of GNSs.
	
\begin{proposition} \cite[Proposition 2.3]{CFPU} \label{GNS special}
	Let $S \subseteq \mathbb{N}_0^d$ be a GNS and $\mathbf{h} \in \operatorname{H}(S)$. Then the unitary extension $S \cup \{ \mathbf{h} \}$ is a GNS if and only if $\mathbf{h} \in \operatorname{EH}(S)$.
\end{proposition}

\begin{lemma} \cite[Lemma 3.4]{CFPU} \label{lemma maximals}
	Let $S$ and $T$ be two GNSs in $ \mathbb{N}_0^d$ such that $S \subsetneq T$. If $\mathbf{h} \in Maximals(T \setminus S)$ (maximal with respect to the natural partial order in $\mathbb{N}_0^d$), then $\mathbf{h} \in \operatorname{EH}(S)$.
\end{lemma}

An important property in numerical semigroups involves the notion of irreducibility, which allows to study decompositions of numerical semigroups as intersections of other ones. This notion was extended to the setting of GNSs in \cite{CFPU}: a GNS $S \subseteq \mathbb{N}_0^d$ is called \emph{irreducible} if it cannot be expressed as an intersection of two GNSs properly containing $S$. Otherwise, $S$ is called \emph{non-irreducible}. 

Among the several characterizations on irreducibility for GNSs given in \cite{CFPU}, we have the following ones.

\begin{proposition} \cite[Propositions 2.5 and 2.6]{CFPU} \label{especial irredutivel}
	Let $S \subseteq \mathbb{N}_0^d$ be a GNS. The following conditions are equivalent:
 \begin{enumerate}
     \item $S$ is irreducible;
     \item $|\operatorname{EH}(S)|=1$;
     \item there exists ${\mathbf f}\in \operatorname{H}(S)$ such that for every ${\mathbf h}\in \operatorname{H}(S)$ with $2{\mathbf h}\neq {\mathbf f}$ we have ${\mathbf f}-{\mathbf h}\in S$.
 \end{enumerate}
\end{proposition}

\begin{lemma} \label{Lema soma Frobenius}
Let $S \in \mathbb{N}_0^d$ be an GNS with $\operatorname{F}(S)$ the Frobenius element. 
\begin{enumerate}
	\item If there exist a component of $\operatorname{F}(S)$ that is odd, then $S$ is irreducible if and only if for all $\mathbf{h},\mathbf{h'}\in\mathbb{N}_0^d$ such that $\mathbf{h}+\mathbf{h'}=\operatorname{F}(S)$, we have that either $\mathbf{h}\in \operatorname{H}(S)$ or $\mathbf{h'}\in H(S)$ (but not both).
	\item If all components of $\operatorname{F}(S)$ are even, $S$ is irreducible if and only if for all $\mathbf{h},\mathbf{h'}\in\mathbb{N}_0^d\setminus\{\operatorname{F}(S)/2\}$ such that $\mathbf{h}+\mathbf{h'}=\operatorname{F}(S)$, we have that either $\mathbf{h}\in \operatorname{H}(S)$ or $\mathbf{h'}\in \operatorname{H}(S)$ (but not both).
\end{enumerate}
\end{lemma}

\begin{proof}
(1) First, suppose $S$ irreducible. Let $\mathbf{h},\mathbf{h'}\in\mathbb{N}_0^d$ such that $\mathbf{h}+\mathbf{h'}=\operatorname{F}(S)$ and suppose that $\mathbf{h} \in \operatorname{H}(S)$. Since $\operatorname{F}(S)$ has a odd component, we have that $\mathbf{h}\neq\mathbf{h'}$ and $2 \mathbf{h} \neq \operatorname{F}(S)$. So, by Proposition \ref{especial irredutivel}, $\mathbf{h'} = \operatorname{F}(S) - \mathbf{h} \in S$.

Conversely, let $\mathbf{h} \in \operatorname{H}(S)$. Since $\operatorname{F}(S)$ has a odd component, we have $2 \mathbf{h} \neq \operatorname{F}(S)$. Then, by hypothesis $\mathbf{h'} = \operatorname{F}(S) - \mathbf{h} \in S$, and so, by Proposition \ref{especial irredutivel}, we have $S$ irreducible.

(2) Using that  $\mathbf{h},\mathbf{h'}\in\mathbb{N}_0^d\setminus\{\operatorname{F}(S)/2\}$ the proof is analogous to the previous item.
\end{proof}

\section{The set of special gaps that preserve the corner element}

In this section we are interested in those special gaps that preserve the corner element of GNSs by unitary extensions. We begin this section by setting this important concept.

	\begin{definition} \label{c special gaps}
		Let $S \subseteq \mathbb{N}_0^d$ be a GNS with corner element $\mathbf{c}$. An element $\mathbf{h} \in \operatorname{EH}(S)$ is called a \emph{corner special gap} (or simply a \emph{$\negc$-special gap}) of $S$ if the unitary extension $S \cup \{ \mathbf{h} \}$ has also corner element $\mathbf{c}$. We denote the set of corner special gaps of $S$ by $\operatorname{CEH}(S)$.
	\end{definition}

This notion was used in \cite{BTT} to construct the set of all GNSs with a prescribed corner element. Notice that for a numerical semigroup $S$, we have $\operatorname{CEH}(S)=\operatorname{EH}(S)\setminus\{\operatorname{F}(S)\}$. Hence, since the corner element coincides with the conductor of numerical semigroups, the elements in $\operatorname{CEH}(S)$ are exactly the special gaps of $S$ whose unitary extensions have $\operatorname{F}(S)$ as the Frobenius number, that is, the Frobenius element is preserved. It is also worth noting that in general the set $\operatorname{CEH}(S)$ might be empty. The next example illustrates the corner special gaps.

\begin{example} \label{ex}
Let $S=\mathbb{N}_0^2\setminus \{(0,1), (1,0), (1,1), (1,2), (3,0)\}$ be the GNS with corner element $(4,3)$ plotted in Figure 1. The elements of $S$ are marked with the black dots, while the elements of $\operatorname{H}(S)$ are the red ones. The elements in $\operatorname{EH}(S)=\{(0,1), (1,1), (1,2), (3,0)\}$ are distinguished with  a black circle. In this case, we have $\operatorname{CEH}(S)=\{(0,1), (1,1)\}$.

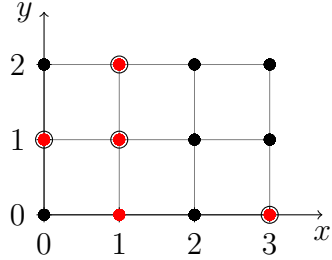
\begin{figure}[h]
\begin{tikzpicture} %[scale=0.6, >=latex, font=\footnotesize]
\draw [help lines] (0,0) grid (3,2);
\draw [<->] (0,2.7) node [left] {$y$} -- (0,0)
-- (3.7,0) node [below] {$x$};
\foreach \i in {0,...,3}
\draw (\i,1mm) -- (\i,-1mm) node [below] {$\i$};
\foreach \i in {0,...,2}
\draw 
(1mm,\i) -- (-1mm,\i) node [left] {$\i$}; 
%\node [below left] at (0,0) {$0$};

\draw (0,1) circle (3.2pt);
\draw (1,1) circle (3.2pt);
\draw (1,2) circle (3.2pt);
\draw (3,0) circle (3.2pt);

\draw [mark=*, color=red] plot (0,1);
\draw [mark=*, color=red] plot (1,0);
\draw [mark=*, color=red] plot (1,1);
\draw [mark=*, color=red] plot (1,2);
\draw [mark=*, color=red] plot (3,0);
\draw [mark=*] plot (0,0);
\draw [mark=*] plot (0,2);
%\draw [mark=*] plot (0,3);
%\draw [mark=*] plot (1,3);
\draw [mark=*] plot (2,0);
\draw [mark=*] plot (2,1);
\draw [mark=*] plot (2,2);
%\draw [mark=*] plot (2,3);
\draw [mark=*] plot (3,1);
\draw [mark=*] plot (3,2);
%\draw [mark=*] plot (3,3);
%\draw [mark=*] plot (4,0);
%\draw [mark=*] plot (4,1);
%\draw [mark=*] plot (4,2);
%\draw [mark=*] plot (4,3);

\end{tikzpicture}
\caption{The GNS with corner element $\negc=(4, 3)$ in Example~\ref{ex}.}\label{fig:GNS}
\end{figure}
\label{exaf}
\end{example}

In what follows, we aim to discuss conditions in order to guarantee that an element is a corner special gap. The following sets will play a role in this purpose due to the characterization of the corner element as in Proposition~\ref{lub}.

\begin{definition}
Let $S \subseteq \mathbb{N}_0^d$ be a GNS with corner $\mathbf{c} = (c_1, \ldots , c_d)$. For each $i=1,\ldots,d$, we define the set $$\operatorname{H}(S)^{(i)} := \{ \mathbf{h} = (h_1,\ldots,h_d) \in \operatorname{H}(S) \ : \  h_i = c_i - 1 \}.$$ For the sake of simplicity, the maximal elements of $\operatorname{H}(S)^{(i)}$ with respect to the partial natural order of $\mathbb{N}_0^d$ will be denoted by $\operatorname{MH}(S)^{(i)}$.
\end{definition}

Observe that all the sets $\operatorname{H}(S)^{(i)}$ (and consequently $\operatorname{MH}(S)^{(i)}$) are non-empty. Moreover, if there exists $i\in \{1,\ldots, d\}$ such that $\operatorname{H}(S)^{(i)} = \{ \mathbf{h} \}$, then $\mathbf{h}$ cannot be a corner special gap. 

% \begin{remark}
% 		If there exist $i$ such that $\operatorname{H}(S)^{(i)} = \{ \mathbf{h} \}$, then $\mathbf{h}$ is not a $c$-special gap.
% 	\end{remark}

	% In the next Proposition, we prove that the elements in the set  $\operatorname{MH}(S)^{(i)} = Maximals (\operatorname{H}(S)^{(i)})$ are special gaps.}
	
	\begin{lemma} \label{maximals especial gap}
	Let $S \subseteq \mathbb{N}_0^d$ be a GNS with corner $\mathbf{c}$. If $\mathbf{z} \in \operatorname{MH}(S)^{(i)}$ for some $i \in \{1,\ldots,d\}$, then $\mathbf{z} \in \operatorname{EH}(S)$. 
	\end{lemma}
 \begin{proof}
     Let $i$ be an index such that $\mathbf{z} \in \operatorname{MH}(S)^{(i)}$. For $\negs\in S^*$, if $\negz'=\negz+\negs \in \operatorname{H}(S)$, as $S$ has corner element $\negc$ and $\negz<\negz'$, we must have $\negz'\in \operatorname{H}(S)^{(i)}$, which contradicts the maximality of $\negz$ in $\operatorname{H}(S)^{(i)}$. Thus, $\negz'\in S$ and therefore $\negz\in \operatorname{PF}(S)$.
     Now, the same argument gives $2\negz \in S$ because $2\negz>\negz$. Hence, we conclude that $\negz\in \operatorname{EH}(S)$.
 \end{proof}
	
	% \begin{proof}
	% Let $\negc=(c_1,\ldots, c_d)$ and suppose that $\mathbf{z} \in Maximals\operatorname{H}(S)^{(i)}$ for some $i \in \{1,\ldots,d\}$. So, $\mathbf{z} = (z_1,\ldots,z_d) \in \operatorname{H}(S)$ and $z_i=c_i-1$. Note that $2z_i \geq c_i - 1$. If $2z_i > c_i - 1$, then from definition of corner we have that $2\mathbf{z} \in S$, and if $2z_i = c_i - 1$, since $2 \mathbf{z} > \mathbf{z}$ and $\mathbf{z} \in \mathcal{M}\operatorname{H}(S)^{(i)}$, we have $2 \mathbf{z} \in S$. Thus, we conclude that $2 \mathbf{z} \in S$. 
	
	% Now, let $\negs \in S^{\ast}$. Suppose that $\mathbf{z}' = \mathbf{z} + \negs \in \operatorname{H}(S)$. Note that, $\mathbf{z}' > \mathbf{z}$ and $z_{i}' = c_i -1 + s_i$. If $s_i > 0$, then from definition of corner $\mathbf{z}'\in S$, contradiction. If $s_i = 0$, then $z_i = c_i - 1$ and it follows that $\mathbf{z}' \in \operatorname{H}(S)^{(i)}$ with $\mathbf{z}' > \mathbf{z}$, contradiction. Thus, $\mathbf{z} + \negs \in S$. Therefore, $\mathbf{z} \in \operatorname{EH}(S)$.
	% \end{proof}

Proposition~\ref{lub} shows that the corner element of a GNS $S$ is completely characterized by $\operatorname{H}(S)$. As $\operatorname{MH}(S)^{(i)}\neq \emptyset$ for all $i\in \{1,\ldots, d\}$, Lemma~\ref{maximals especial gap} yields that the corner element is determined by its special gaps.

%In \cite{BTT}, it was shown that the corner of a GNS $S$ is given by $\negc = \lub(\operatorname{H}(S)) + \neg1 = \lub(\operatorname{PF}(S)) + \neg1$. As a consequence of Proposition \ref{maximals especial gap}, we obtain the following.
%
\begin{corollary}\label{lub lacunas especiais}
Let $S\subseteq \mathbb{N}_0^d$ be a GNS with positive genus and corner element $\negc$. Then
$$\negc = \lub(\operatorname{EH}(S)) + \neg1.$$
\end{corollary}

\begin{remark}\label{remark maximal gaps}
    The set $\operatorname{EH}(S)$ is not the smallest subset of $\operatorname{H}(S)$ that allows us to describe the corner element of $S$ in the fashion of the previous result. Actually, its subset
    $\mathcal{B}=\bigcup_{i=1}^d \operatorname{MH}(S)^{(i)}$ also satisfies $\negc = \lub(\mathcal{B}) + \neg1$.
\end{remark}

\begin{corollary} \label{lemma auxiliar}
    Let $S\subseteq \mathbb{N}_0^d$ be a GNS with corner element $\negc$. If $S$ has a unique maximal element $\negh$ in $\operatorname{EH}(S)$ with respect to the natural partial order of $\mathbb{N}_0^d$, then $S$ is Frobenius with $\operatorname{F}(S)=\negh$.
\end{corollary}
\begin{proof}
    In this case, $\operatorname{lub}(\operatorname{EH}(S))=\negh$. Hence, we obtain from Corollary~\ref{lub lacunas especiais} that $\negh=\negc-\mathbf{1}\in \operatorname{EH}(S)$, which implies that $S$ is Frobenius and  $\operatorname{F}(S)=\negh$.
\end{proof}

%Observar - CorolÃ¡rio 3.5 e 3.7 (texto)

The subsequent results gives sufficient conditions so that a special gap preserves the corner element by unitary extensions.

\begin{proposition}
		Let $S \subseteq \mathbb{N}_0^d$ be a GNS with corner element $\mathbf{c}$. If $\mathbf{h}\in \operatorname{EH}(S)$ satisfies $\mathbf{h} < \mathbf{g}$ for some $\negg\in \operatorname{H}(S)$, then $\mathbf{h}\in \operatorname{CEH}(S)$.
	\label{c-especial dois elementos}
\end{proposition}
\begin{proof}
         Let $T=S\cup \{\mathbf{h}\}$. Since $T$ is a GNS by Proposition~\ref{GNS special}, it follows from Lemma~\ref{corner subset} that $\mathbf{c}(T) \leq \mathbf{c}$. Now, observe that $\negh$ cannot belong to any $\operatorname{MH}(S)^{(i)}$ because $\mathbf{h} < \mathbf{g}$. Hence, as $\operatorname{H}(T)=\operatorname{H}(S)\setminus\{\negh\}$, we have $\operatorname{MH}(S)^{(i)} \subseteq \operatorname{H}(T)$ for all $i\in \{1, \ldots, d\}$, and consequently $\mathcal{B}=\bigcup_{i=1}^d \operatorname{MH}(S)^{(i)}\subseteq \operatorname{H}(T)$. In view of  Remark~\ref{remark maximal gaps} and Proposition~\ref{lub}, we conclude that $\negc\leq \mathbf{c}(T)$. Therefore, we obtain $\mathbf{c}(T)=\negc$, which implies $\mathbf{h} \in \operatorname{CEH}(S)$.
        %Since $\operatorname{H}(S)^{(i)}\neq \emptyset$ for all $i\in \{1,\ldots, d\}$, let us consider $\negz_i \in \operatorname{H}(S)^{(i)}$ for $i=1,\ldots, d$.
\end{proof}

	\begin{proposition} \label{lemma hi}
		Let $S \subseteq \mathbb{N}_0^d$ be a GNS with corner element $\mathbf{c}$. If $\mathbf{h} \in \operatorname{EH}(S)$ satisfies $\negh \notin \bigcup_{i=1}^d \operatorname{H}(S)^{(i)}$, then $\mathbf{h}\in \operatorname{CEH}(S)$.
	\end{proposition}
	\begin{proof}
		Since the GNS $T=S\cup \{\mathbf{h}\}$ contains $S$,  Lemma~\ref{corner subset} yields $\mathbf{c}(T) \leq \mathbf{c}$. As $\negh \notin \operatorname{H}(S)^{(i)}$ for all $i \in \{1,\ldots,d\}$, we have that $\negh$ cannot belong to any of the sets $\operatorname{MH}(S)^{(i)}$. Consequently, we obtain $\operatorname{MH}(S)^{(i)} \subseteq \operatorname{H}(T)$ and the same argument of the last proof guarantees that $\mathbf{h} \in \operatorname{CEH}(S)$.
          %for all $i\in \{1, \ldots, d\}$, and thus $\mathcal{B}=\bigcup_{i=1}^d \operatorname{MH}(S)^{(i)}\subseteq \operatorname{H}(T)$. Hence, we conclude that $\negc\leq \mathbf{c}(T)$ by Remark~\ref{remark maximal gaps} and Proposition~\ref{lub}. Therefore, $\mathbf{h} \in \operatorname{CEH}(S)$ because $\mathbf{c}(T)=\negc$.
	\end{proof}

	\begin{proposition} \label{lemma H(S)i}
		Let $S \subseteq \mathbb{N}_0^d$ be a GNS with corner element $\mathbf{c}$. If $\mathbf{h} \in \operatorname{EH}(S) \setminus \operatorname{CEH}(S)$, then there exists $i\in \{1,\ldots, d\}$ such that $\operatorname{H}(S)^{(i)} = \{ \mathbf{h} \}$.
	\end{proposition}
	\begin{proof}
		Let $\mathbf{h}\in \operatorname{EH}(S) \setminus \operatorname{CEH}(S)$ and consider the GNS $T=S\cup \{\negh\}$. As $\negh\notin \operatorname{CEH}(S)$, it follows from Proposition~\ref{lemma hi} that there exists $i \in \{1,\ldots,d\}$ such that $\negh\in \operatorname{H}(S)^{(i)}$. Now, suppose that for all $i\in \{1,\ldots, d\}$ such that $\negh\in \operatorname{H}(S)^{(i)}$ we have $|\operatorname{H}(S)^{(i)}| \geq 2$. Hence, as $\operatorname{H}(T)=\operatorname{H}(S)\setminus \{\negh\}$, $\negc(T)$ coincides with $\negc$ by Proposition~\ref{lub}, since  $\operatorname{H}(T)^{(i)}$ contains an element of $\operatorname{H}(S)^{(i)}$ for all $i\in \{1,\ldots, d\}$. It implies that $\mathbf{h} \in \operatorname{CEH}(S)$, a contradiction. 
	\end{proof}

 With the previous results, we have the following about the set of corner special gaps.

 	\begin{proposition} \label{prop c especial}
%		Let $S \subset \mathbb{N}_0^d$ be a GNS. If $\{\mathbf{h}_1, \ldots , \mathbf{h}_{d+1} \} \subset \operatorname{EH}(S)$, then $\mathbf{h}_j \in \operatorname{CEH}(S)$ for some $j \in \{1,\ldots,d+1\}$.
		
Let $S \subseteq \mathbb{N}_0^d$ be a GNS such that $|\operatorname{EH}(S)|\geq d + 1$. Then $$\operatorname{CEH}(S) \neq \emptyset.$$
\end{proposition}
\begin{proof}
		Let $\{\mathbf{h}_1, \ldots , \mathbf{h}_{d+1} \} \subseteq \operatorname{EH}(S)$. If there exists $j\in\{1,\ldots, d+1\}$ such that $\negh_j\notin \operatorname{H}(S)^{(i)}$ for all $i\in \{1,\ldots,d\}$, then $\mathbf{h}_j \in \operatorname{CEH}(S)$ by Proposition~\ref{lemma hi}. Otherwise, we suppose that $\{\mathbf{h}_1, \ldots , \mathbf{h}_{d+1} \} \cap \operatorname{CEH}(S) = \emptyset$. Hence, Proposition~\ref{lemma H(S)i} yields that for all $j\in \{1,\ldots, d+1\}$ there exists $i_j\in \{1,\ldots, d\}$ such that $\operatorname{H}(S)^{(i_j)}=\{\negh_j\}$. But this gives a contradiction since by the Pigeonhole principle there exists $k\in \{1,\ldots, d\}$ such that $|\operatorname{H}(S)^{(k)}|\geq 2$. Therefore, in every case, we obtain $\mathbf{h}_j \in \operatorname{CEH}(S)$ for some $j \in \{1,\ldots,d+1\}$, which proves the result.
\end{proof}

Successive applications of the previous result yield the following.

\begin{corollary} \label{CEHk}
Let $S \subseteq \mathbb{N}_0^d$ be a GNS such that $|\operatorname{EH}(S)| \geq d + k$, for some positive integer $k$. Then $$|\operatorname{CEH}(S)| \geq k\;.$$
\end{corollary}

In the remainder of this section, we discuss the relation between irreducibility and the set of corner special gaps. Notice that in the context of numerical semigroups ($d=1$), as $\operatorname{CEH}(S)=\operatorname{EH}(S)\setminus \{\operatorname{F}(S)\}$, Proposition~\ref{especial irredutivel} establishes that $S$ is irreducible if and only if $\operatorname{CEH}(S)=\emptyset$.

\begin{proposition}
	Let $S\subseteq \mathbb{N}^d$ be a GNS with corner element $\negc$. If $S$ is a Frobenius GNS with $\operatorname{CEH}(S)=\emptyset$, then $S$ is irreducible. 
\end{proposition}
\begin{proof}
	Suppose that $S$ is a Frobenius GNS which is not irreducible. It follows from Proposition~\ref{especial irredutivel} that $|\operatorname{EH}(S)|\geq 2$. As $\operatorname{F}(S)=\negc-\mathbf{1}\in \operatorname{EH}(S)$ by Lemma~\ref{maximals especial gap}, there exists $\negh \in \operatorname{EH}(S)$ with $\negh<\operatorname{F}(S)$. Hence, we obtain from Proposition~\ref{c-especial dois elementos} that $\negh\in \operatorname{CEH}(S)$, contradicting $\operatorname{CEH}(S)=\emptyset$. Therefore, $S$ is irreducible.
\end{proof}

Differently from the numerical semigroup case, the converse of the previous result does not hold in general as we can see in the following example.

\begin{example}
Let $S=\mathbb{N}^2\setminus\{(1,0), (1,1), (1,2), (2,0), (2,1), (4,0), (4,1), (4,2)\}$.
Observe that $S$ has corner element $(5, 3)$ and is irreducible because $\operatorname{EH}(S)=\{(4,2)\}$. Moreover, the Frobenius element $(4, 2)$ of $S$ is a corner special gap of $S$ since $\operatorname{H}(S)^{(1)}=\{(4, 0), (4, 1), (4, 2)\}$ and $\operatorname{H}(S)^{(2)}=\{(1, 2), (4, 2)\}$.
\end{example}
	
	\section{Atomic generalized numerical semigroups}

  In this section, we introduce the notion of atoms of the set of GNSs with fixed corner element. This extends the concept of atomic numerical semigroups of \cite{Rosales2} to the setting of GNSs. We explore the outcomes of this notion and its differences for the numerical semigroup case.
  
	For $\mathbf{c} \in  \mathbb{N}_0^d$, let us denote
	$$
	\mathcal{F}(\mathbf{c}) := \{ S \subseteq  \mathbb{N}_0^d \ : \ S \mbox{ is a GNS with corner element } \mathbf{c}  \}.
	$$
	
    It can be easily checked that $\mathcal{F}(\mathbf{c})$ is closed with respect to the intersection of sets. Hence, an element $S$ of $\mathcal{F}(\mathbf{c})$ is called an \textit{atom} if it cannot be written as an intersection of two GNSs in $\mathcal{F}(\negc)$ properly containing $S$. A GNS $S$ with corner element $\negc$ is called \textit{atomic} if it is an atom of $\mathcal{F}(\negc)$. 
    
    It follows from the definition of atom that every GNS in $\mathcal{F}(\mathbf{c})$ can be expressed as an intersection of finitely many atomic GNSs in $\mathcal{F}(\mathbf{c})$. This is because the set of GNSs that contain properly $S$ is finite, since $\mathbb{N}_0^d\setminus S$ has finitely many elements. We point out that every irreducible GNS is atomic, but the converse does not hold true, as we will see throughout this section. In particular, the notion of atomic GNS extends the concept of irreducible GNS, in somehow. %We observe that atoms in the framework for numerical semigroups have been considered in \cite{Rosales2, Rosales3}.

We now study the Frobenius GNSs that are atoms. If it is irreducible, then $|\operatorname{EH}(S)| = 1$ (cf. Proposition \ref{especial irredutivel}). The next result deals with the atomic and non-irreducible Frobenius GNSs.

%We will refer to a Frobenius GNS which is an atom and non-irreducible by ANI Frobenius GNS.

	\begin{proposition} \label{especial ANI Frobenius}
		Let $S \subseteq \mathbb{N}_0^d$ be an atomic and non-irreducible Frobenius GNS. Then $|\operatorname{EH}(S)|=2$.
	\end{proposition}
	\begin{proof}
		Let $\operatorname{F}(S)$ be the Frobenius element of $S$. It follows from Proposition \ref{especial irredutivel} that $|\operatorname{EH}(S)| \geq 2$. Suppose that $|\operatorname{EH}(S)| > 2$. So, there exist $\mathbf{h}_1, \mathbf{h}_2 \in \operatorname{EH}(S) \setminus \{\operatorname{F}(S)\}$ such that $\negh_1 < \operatorname{F}(S)$ and $\negh_2 < \operatorname{F}(S)$. It follows from Lemma \ref{c-especial dois elementos} that $\negh_1$ and $\negh_2$ are corner special gaps of $S$. Since $S = (S \cap \{\negh_1\}) \cup (S \cap \{\negh_2\})$, it contradicts $S$ to be atomic.
%Note that, $S_1 =S \cup \{\mathbf{h}_1\}$ and $S_2=S \cup \{\mathbf{h}_2\}$ are GNSs with the same corner of $S$, since $\mathbf{f} \in \operatorname{H}(S_1) \cap \operatorname{H}(S_2)$. Thus, we get a contradiction. Therefore, $|\operatorname{EH}(S)|=2$.
	\end{proof}

%\begin{remark}
%Note that for $S=\mathbb{N}^2\setminus\{(2,2),(4,2),(2,4),(4,4)\}$
%NÃO Ã GNS
%\end{remark}

We recall that for the case $d = 1$ we have the following: $S$ is an atomic and non-irreducible numerical semigroup if and only if $|\operatorname{EH}(S)| = 2$ (cf. \cite{Rosales2}). However, in the general case, the converse of Proposition \ref{especial ANI Frobenius} is not true, even for Frobenius GNSs. 

\begin{example}\label{exemplo nao atomo}
    Let $$S=\mathbb{N}_0^2\setminus\{(1,0), (1,1), (2,0), (2,1), (2,2), (2,3), (3,0), (3,1), (3,2), (3,3) \}$$
    be the Frobenius GNS with corner element $(4,4)$. In this case, we have $\operatorname{EH}(S) =\operatorname{CEH}(S) = \{(2,3), (3,3)\}$ and therefore we can write $S = (S \cup \{(2,3)\}) \cap (S \cup \{(3,3)\})$. Hence, $S$ is not an atom.
\end{example}

%\textcolor{red}{[Gancho com o caso numÃ©rico + ComentÃ¡rio + Troca de ordem entre 4.1 e 4.2? + Caso geral: cota nÃ£o vale]}

The next result relates the concept of atoms with the the set of special gaps of a GNS by a necessary condition. %and it also generalizes Lemma 2 of \cite{Rosales2}

\begin{proposition} \label{cota superior especial}
Let $S \subseteq \mathbb{N}_0
^d$ be an atomic GNS. Then $$|\operatorname{EH}(S)| \leq d+1.$$
\end{proposition}

\begin{proof}
Suppose that $|\operatorname{EH}(S)| \geq d + 2$. By Corollary \ref{CEHk}, it follows that $|\operatorname{CEH}(S)| \geq 2$. Thus, there exist $\mathbf{h}_1, \mathbf{h}_2 \in \operatorname{CEH}(S)$ and we can write $S = (S \cup \{\negh_1\}) \cap (S \cup \{\negh_2\})$ and we conclude that $S$ is not an atom, which is a contradiction.
%Suppose that $|\operatorname{EH}(S)| > d+1$. So, there are $\mathbf{h}^{(1)}, \ldots , \mathbf{h}^{(d+1)}, \mathbf{h}^{(d+2)} \in \operatorname{EH}(S)$. By Proposition \ref{prop c especial}, \textcolor{red}{[vale a pena especificar o procedimento de (1) tomar um subconjunto com $d+1$ que vai conter uma c-especial e, depois de excluir essa c-especial, ainda haverÃ¡ um conjunto com $d+1$ elementos em que haverÃ¡ mais uma c-especial, diferente da primeira]} there exist $\mathbf{h}^{(j_1)}, \mathbf{h}^{(j_2)} \in \operatorname{CEH}(S)$, for some $j_1, j_2 \in \{1,\ldots,d+2\}$, with $j_1 \neq j_2$. Hence, we can write $S = (S \cup \{\negh^{(j_1)}\}) \cap (S \cup \{\negh^{(j_2)}\})$ and we conclude that $S$ is not an atom, which is a contradiction.
%So, by Proposition \ref{prop atomo c especial} we have a contradiction, since $S$ is an atom.
\end{proof}

Differently from the numerical semigroup case, the converse does not hold true, as we can in Example~\ref{exemplo nao atomo}. Despite the structure of special gaps does not characterize entirely the property of atomicity in GNSs, the subset of corner special gaps does. The characterization we provide for atomic GNSs in the sequence generalizes that given in Lemma 2 of \cite{Rosales2} for numerical semigroups.
	
	\begin{theorem} \label{prop atomo c especial}
		Let $S\subseteq \mathbb{N}_0^d$ be a GNS with corner element $\negc$. Then $S$ is an atom if and only if $|\operatorname{CEH}(S)| \leq 1$.
	\end{theorem}
	\begin{proof}
		First, let $S$ be an atom and suppose that  $|\operatorname{CEH}(S)| \geq 2$. Thus, there exist $\mathbf{h}_1, \mathbf{h}_2 \in \operatorname{CEH}(S)$, with $\mathbf{h}_1 \neq \mathbf{h}_2$. By Proposition \ref{GNS special}, we have that $S_1 = S \cup \{\mathbf{h}_1  \}$ and $S_2 = S \cup \{ \mathbf{h}_2\}$ are GNSs, and from the definition of  corner special gap, it follows that both $S_1$ and $S_2$ have corner element $\mathbf{c}$. So, since $S = S_1 \cap S_2$, we have a contradiction. Therefore, $|\operatorname{CEH}(S)| \leq 1$.
		
		Conversely, suppose that $|\operatorname{CEH}(S)| \leq 1$. If $S$ is not an atom, then there exist $S_1$ and $S_2$ GNSs with corner element $\negc$ such that $S = S_1 \cap S_2$, with $S \neq S_1$ and $S \neq S_2$. From Lemma \ref{lemma maximals}, it follows that there exist $\mathbf{h}_1 \in  Maximals(S_1 \setminus S) \cap \operatorname{EH}(S)$ and $\mathbf{h}_2 \in  Maximals(S_2 \setminus S) \cap \operatorname{EH}(S)$. Now, by Corollary~\ref{lemma corner subset} and Proposition \ref{GNS special}, we have that $S \cup \{\mathbf{h}_1\} \subseteq S_1$ and $S \cup \{\mathbf{h}_2\} \subseteq S_2$ are both GNSs with corner element $\mathbf{c}$, and so we get  $ \mathbf{h}_1,  \mathbf{h}_2 \in \operatorname{CEH}(S)$, a contradiction. And the result follows.
	\end{proof}
	
%{\color{blue} porque $\mathbf{h}_1$ e $\mathbf{h}_2$ sÃ£o distintos?} se fossem iguais, entÃ£o estariam em S, algo que nÃ£o ocorre.

%The next example shows that the bound on $|\operatorname{EH}(S)|$ obtained in Proposition \ref{cota superior especial} is sharp. %It is also an application of Theorem \ref{prop atomo c especial}.

\begin{example}
Let $S=\mathbb{N}_0^2\setminus\{(0,1),(1,0),(1,1),(2,0),(2,1),(2,2),(3,1)\}$. In this case, $\operatorname{EH}(S)=\{(2,1),(2,2),(3,1)\}$. Note that $S$ is an atom because $\operatorname{CEH}(S)=\{(2,1)\}$, and so the bound on $|\operatorname{EH}(S)|$ obtained in Proposition \ref{cota superior especial} is sharp.
\end{example}

%Buscar por famÃ­lias de GNS formadas por Ã¡tomos atingindo a cota superior.

%\textcolor{red}{[inserir um texto sobre fixar o corner e estudar questÃµes de maximalidade; precisaremos de alguns objetos que levam atÃ© esse objetivo; nos numÃ©ricos, os Ã¡tomos possuem uma questÃ£o de maximalidade]}

% From now on, we study the set GNSs with a fixed corner. For numerical semigroups, it is well known that atoms satisfy a maximal condition and now we intend to deduce a similar property to the general case. Rosales \cite{Rosales2} considered the set of numerical semigroups with given gaps $g_1, g_2, \ldots,$ $g_n$, namely $L(g_1, g_2, \ldots, g_n)$ and the set of maximal numerical semigroups (with respect to set inclusion) that have $g_1, g_2, \ldots, g_n$ as gaps, which is denoted by $\mathcal{M} L (g_1, g_2, \ldots, g_n)$ and obtained the fo\-llowing result.

% \begin{lemma} [\cite{Rosales2}, Lemma 3] \label{lemma rosales} Let $S$ be a numerical semigroup and $\{h_1, h_2, \ldots, h_n\} \subseteq \operatorname{H}(S)$. Then $S \in \mathcal{M} L(h_1, h_2, \ldots, h_n)$ if, and only if, $\operatorname{EH}(S) \subseteq \{h_1, h_2, \ldots, h_n\}$.
% 	\label{rosales} 
% \end{lemma}

%\textcolor{red}{Dizer que Ã© possÃ­vel utilizar o conceito de $L$ e $\mathcal{M}L$ nos GNSs e isso Ã© feito na seÃ§Ã£o 5. Todavia, fixar lacunas acaba nÃ£o garantindo que o corner estÃ¡ fixado.}

Atomic numerical semigroups can be characterized in terms of a maximal property in the set of numerical semigroups having two fixed gaps (cf. Theorem 4 of \cite{Rosales2}). In what follows, we bring this discussion to the setting of GNSs, involving it with the corner element. To this end, given positive integers $h<c$, we consider $\mathcal{F}(c; h)$ the set of numerical semigroups $S$ with fixed conductor $c$ and $S\cap \{h\}=\emptyset$, and $\mathcal{M} \mathcal{F}(c; h)$ the set of its maximal elements with respect to the inclusion. Using the same notation of \cite{Rosales2}, $L(g_1, g_2)$ means the set of numerical semigroups $S$ such $S\cap \{g_1,g_2\}=\emptyset$ and $\mathcal{M}L(g_1, g_2)$ its maximal elements, we can state the following relation.

% In order to consider numerical semigroups with a fixed conductor, we define the set $\mathcal{F}(c; h) := \{S \text{ numerical semigroup }: c(S) = c \text{ and } h \in \operatorname{H}(S)\}$ and the set $\mathcal{M} \mathcal{F}(c; h)$ as the set of maximal elements of $\mathcal{F}(c; h)$ with respect to the inclusion. We obtain the following result. \highlight{Ver}

\begin{lemma} \label{prop motivacao}
	Let $g_1$ and $g_2$ be positive integers such that $g_1<g_2$. Then $$\mathcal{M} L(g_1, g_2) = \mathcal{M} \mathcal{F}(g_2+1; g_1).$$
\end{lemma}
\begin{proof}
	Let $S \in \mathcal{M} L(g_1 , g_2)$. By \cite[Lemma 3]{Rosales2}, we have that $\operatorname{EH}(S) \subseteq \{g_1, g_2\}$. Since the Frobenius number is a special gap of $S$ and $g_1< g_2$, we conclude that the conductor of $S$ is $g_2+1$. Now we prove that $S$ is maximal in the set $\mathcal{F}(g_2+1; g_1)$. If $T \in \mathcal{F}(g_2+1; g_1)$ is such that $S \subseteq T$, then $g_1, g_2 \notin T$, and from the maximality of $S$ in $L(g_1, g_2)$, we conclude that $S = T$. \\
	On the other hand, if $S \in \mathcal{M} \mathcal{F} (g_2+1; g_1)$, then $g_1, g_2 \notin S$ and $S \in L(g_1, g_2)$. Now we prove that $S$ is maximal in the set $L(g_1, g_2)$. If there exists $T \in \mathcal{M} L(g_1, g_2)$ such that $S \subseteq T$, then $g_1, g_2 \notin T$. In particular, the conductor of $T$ is $g_2+1$ and $T \in \mathcal{M} \mathcal{F} (g_2+1; g_1)$. From the maximality of $S$ in $\mathcal{F}(g_2+1; g_1)$, we conclude that $S = T$.
\end{proof}

Lemma~\ref{prop motivacao} motivates us to introduce the following: for $\negc, \negh \in \mathbb{N}_0^d$ with $\negh\leq \negc-\mathbf{1}$, define $$\mathcal{F}(\negc; \negh) := \{S\in \mathcal{F}(\negc) \ : \ \negh \in \operatorname{H}(S)\}.$$ Denote by $\mathcal{M} \mathcal{F}(\negc; \negh)$ the set of maximal elements (with respect to the inclusion) of $\mathcal{F}(\negc; \negh)$. Hence, we have the following property.

\begin{proposition} \label{CEHunit}
	Let $S\subseteq \mathbb{N}_0^d$ be a GNS with corner element $\negc$ and $\negh\in \operatorname{H}(S)$. Then $\operatorname{CEH}(S) \subseteq \{\negh\}$ if and only if $S\in \mathcal{M} \mathcal{F}(\negc; \negh)$.
\end{proposition}

\begin{proof}
	$(\Rightarrow)$ Suppose that there exists $T\in \mathcal{F}(\negc)$ such that $S\subsetneq T$ and $\negh\in \operatorname{H}(T)$. By Lemma \ref{lemma maximals}, for $\negx\in Maximals(T\setminus S)$ we have that $\negx\in \operatorname{EH}(S)$. So, as $S\subseteq S\cup \{\negx\}\subseteq T$ and $S, T\in \mathcal{F}(\negc)$, we have $S\cup \{\negx\}\in \mathcal{F}(\negc)$ by Corollary~\ref{lemma corner subset}. It implies that $\negx \in \operatorname{CEH}(S)$, contradicting the hypotheses that $\operatorname{CEH}(S) \subseteq \{\negh\}$. \\
	$(\Leftarrow)$ Suppose that there exists $\negx \in \operatorname{CEH}(S)$ such that $\negx \neq \negh$. Then, $S \subseteq S \cup \{\negx\}$ are GNSs with corner element $\negc$ and have $\negh$ as a gap. Therefore, we obtain that $S \notin \mathcal{M} \mathcal{F}(\negc; \negh)$.
\end{proof}

\begin{remark} \label{CEHvazio}
	Observe that, in the conditions of the previous result, $\operatorname{CEH}(S)=\emptyset$ is equivalent to $S$ being maximal in $\mathcal{F}(\negc)$ with respect to the inclusion.
\end{remark}

%\textcolor{red}{[Colocar um texto conectando Prop 4.7 e Rem 4.8; ajustar o Remark] \\
%The case $\operatorname{CEH}(S) = \emptyset$...}

In our next result, we state a characterization for atomic GNSs in terms of a maximal property among GNSs with fixed corner. It extends what is known for numerical semigroups in Theorem 4 of \cite{Rosales2}. %, when the corner is fixed.

\begin{theorem} \label{atom MF}
	Let $S\subseteq\mathbb{N}_0^d$ be a GNS with corner element $\negc$ and positive genus. Then $S$ is an atom if and only if $S \in \mathcal{M} \mathcal{F}(\negc; \negh)$ for some $\negh\in \mathbb{N}_0^d$.
\end{theorem}
\begin{proof}
	$(\Rightarrow)$ If $S$ is an atom, then, by Theorem \ref{prop atomo c especial}, we conclude that $|\operatorname{CEH}(S)| \leq 1$. If $\operatorname{CEH}(S) = \emptyset$, then Remark \ref{CEHvazio} guarantees that $S$ is maximal in $\mathcal{F}(\negc)$. In particular, it is maximal in the subset of GNSs in $\mathcal{F}(\negc)$ such that $\negh$ is a gap, for any $\negh$. If $\operatorname{CEH}(S)$ has one element $\negh$, then Proposition \ref{CEHunit} ensures that $S \in \mathcal{M} \mathcal{F}(\negc; \negh)$. \\
	$(\Leftarrow)$ If $S \in \mathcal{M} \mathcal{F}(\negc; \negh)$ for some $\negh\in \mathbb{N}_0^d$, then it follows from Proposition~\ref{CEHunit} that $\operatorname{CEH}(S)\subseteq \{\negh\}$. It implies that $|\operatorname{CEH}(S)|\leq 1$ and hence $S$ is atom by Theorem~\ref{prop atomo c especial}.
\end{proof}
%PENSAR NA QUESTÃO COROLÃRIO/TEOREMA

%\begin{remark} \label{remark atom MF}
%We note that in the previous result the condition for some $\negh\in \mathbb{N}_0^d$ can be change to the condition for some $\negh\in \operatorname{EH}(S)$ and the proof is similar.
%\end{remark}

We can construct all GNSs $S$ in the set $\mathcal{F}(\negc; \negh)$ and thus compute the elements in $\mathcal{M} \mathcal{F}(\negc; \negh)$ by a standard
procedure. Starting with $S=\mathcal{O}(\negc)$, the ordinary GNS with corner element $\negc$, we consider the unitary extensions $S\cup \{\negx\}$, where $\negx\in \operatorname{CEH}(S)$ and $\negx\neq \negh$. For each of these new GNSs, we repeat the process to obtain new ones arising by adding the corner special gaps which are
different from $\negh$. This procedure ends up after a finite number of steps since $\mathcal{F}(\negc; \negh)$ is a finite set. The elements of $S\in \mathcal{M} \mathcal{F}(\negc; \negh)$ are exactly those in $\mathcal{F}(\negc; \negh)$ satisfying either $\operatorname{CEH}(S)=\emptyset$ or $\operatorname{CEH}(S)=\{\negh\}$.

\begin{example} \label{exemplo MF}
Let $S=\mathcal{O}((3,2))$ be the ordinary GNS of $\mathbb{N}_0^2$ with corner element $\negc=(3,2)$ and consider $\negh=(2,1)$. Let us compute the sets $\mathcal{F}(\negc; \negh)$ and $\mathcal{MF}(\negc; \negh)$ by the previous procedure. As $\operatorname{CEH}(S)=\{(0,1), (1,1), (2,0), (2,1)\}$, adding the elements of $\operatorname{CEH}(S)$ which are different from $\negh$, we have the GNSs in $\mathcal{F}(\negc; \negh)$: 
\begin{itemize}
    \item $S_1=S\cup \{(0,1)\}$ with $\operatorname{CEH}(S_1)=\{(1,1), (2, 1)\}$;
    \item $S_2=S\cup \{(1,1)\}$ with $\operatorname{CEH}(S_2)=\{(0,1), (2,0), (2,1)\}$;
    \item $S_3=S\cup \{(2,0)\}$ with $\operatorname{CEH}(S_3)=\{(1,0),(1,1)\}$;
\end{itemize}
For each of these GNSs, adding the corner special gaps which are different from $\negh$, we obtain the following GNSs in $\mathcal{F}(\negc; \negh)$: 
\begin{itemize}
    \item $S_4=S_1\cup\{(1,1)\}$ with $\operatorname{CEH}(S_4)=\emptyset$;
    \item $S_5=S_2\cup\{(0,1)\}=S_4$;
    \item $S_6=S_2\cup \{(2,0)\}$ with $\operatorname{CEH}(S_6)=\emptyset$;
    \item $S_7=S_3\cup\{(1,0)\}$ with $\operatorname{CEH}(S_7)=\emptyset$;
    \item $S_8=S_3\cup \{(1,1)\}=S_6$. 
\end{itemize}
Hence, $\mathcal{F}(\negc; \negh)=\{S,S_1,S_2,S_3, S_4, S_6, S_7\}$ and therefore $\mathcal{MF}(\negc; \negh)=\{S_4,S_6,S_7\}$, whose elements are the irreducible GNSs with corner element $\negc$ (see Proposition~\ref{prop irreducible MF}).
\end{example}

Observe in Example~\ref{exemplo MF} that the aforementioned procedure may produce repeated GNSs in its steps. In order to avoid a such redundant computations, one can implement a simple modification in the procedure by using a monomial order in $\mathbb{N}_0^d$, in the same way as used in \cite{BTT} to organize the GNSs with a fixed corner element in a rooted tree. Given a monomial order $\preceq$ in $\mathbb{N}_0^d$, once we have obtained the unitary extensions of the ordinary GNS $\mathcal{O}(\negc)$ by adding its corner special gaps different from $\negh$, instead of repeating this process for each of the these GNSs $T$, we just add the corner special gaps different from $\negh$ which are smaller than the elements $\negs$ in $T^*$ satisfying $\negs\leq \negc-\mathbf{1}$ with respect to the fixed monomial order $\preceq$ (cf. Section 5 of \cite{BTT}). 

Figure 1 illustrates the GNSs of Example~\ref{exemplo MF} given by this method by using the lexicographic order in $\mathbb{N}_0^2$. The elements marked in red and black bullets are respectively gaps and non-gaps of the GNSs in that example. Notice that there is no repeated GNSs with this procedure.

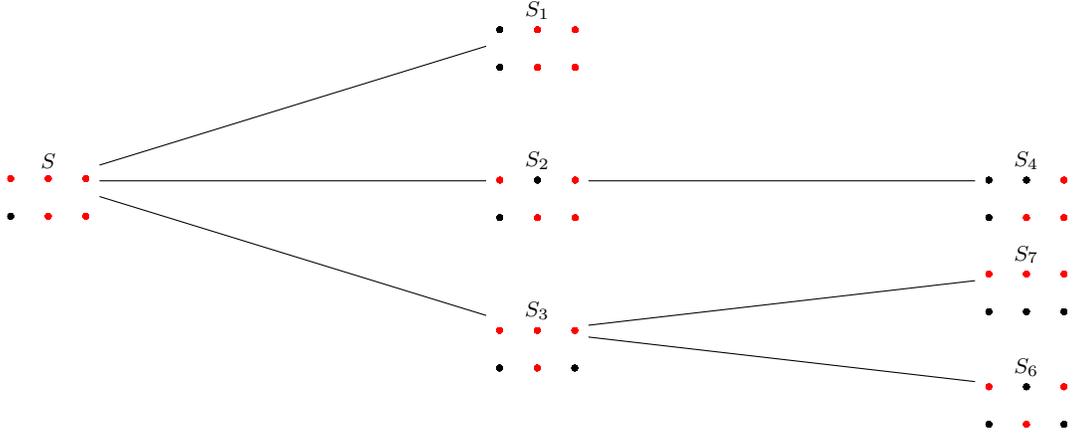
\begin{figure}[htp] \label{fig2}
		\begin{center}
			\begin{tikzpicture}[grow=right, level distance=6.5cm,
  level 1/.style={sibling distance=2cm},
  level 2/.style={sibling distance=1.5cm}]
			\node {\begin{tikzpicture}[scale=.5]
                    \draw (1,1) node [above] {\tiny $S$};
				\draw [mark=*, color=red] plot (1,0);
				\draw [mark=*, color=red] plot (0,1);
				\draw [mark=*] plot (0,0);
				\draw [mark=*, color=red] plot (2,0);
				\draw [mark=*, color=red] plot (1,1);
				\draw [mark=*, color=red] plot (2,1);
				\end{tikzpicture}} 
                child { node {\begin{tikzpicture}[scale=.5]
                        \draw (1,1) node [above] {\tiny $S_3$};
					\draw [mark=*, color=red] plot (1,0);
					\draw [mark=*, color=red] plot (0,1);
					\draw [mark=*] plot (0,0);
					\draw [mark=*] plot (2,0);
					\draw [mark=*, color=red] plot (1,1);
					\draw [mark=*, color=red] plot (2,1);
					\end{tikzpicture}}
				child { node {\begin{tikzpicture}[scale=.5]
                            \draw (1,1) node [above] {\tiny $S_6$};
						\draw [mark=*, color=red] plot (1,0);
						\draw [mark=*, color=red] plot (0,1);
						\draw [mark=*] plot (0,0);
						\draw [mark=*] plot (2,0);
						\draw [mark=*] plot (1,1);
						\draw [mark=*, color=red] plot (2,1);
						\end{tikzpicture}} }
                    child { node {\begin{tikzpicture}[scale=.5]
                            \draw (1,1) node [above] {\tiny $S_7$};
						\draw [mark=*] plot (1,0);
						\draw [mark=*, color=red] plot (0,1);
						\draw [mark=*] plot (0,0);
						\draw [mark=*] plot (2,0);
						\draw [mark=*, color=red] plot (1,1);
						\draw [mark=*, color=red] plot (2,1);
						\end{tikzpicture}} }}
			child { node {\begin{tikzpicture}[scale=.5]
                        \draw (1,1) node [above] {\tiny $S_2$};
					\draw [mark=*, color=red] plot (1,0);
					\draw [mark=*, color=red] plot (0,1);
					\draw [mark=*] plot (0,0);
					\draw [mark=*, color=red] plot (2,0);
					\draw [mark=*] plot (1,1);
					\draw [mark=*, color=red] plot (2,1);
					\end{tikzpicture}} 
				child { node {\begin{tikzpicture}[scale=.5]
                            \draw (1,1) node [above] {\tiny $S_4$};
						\draw [mark=*, color=red] plot (1,0);
						\draw [mark=*] plot (0,1);
						\draw [mark=*] plot (0,0);
						\draw [mark=*, color=red] plot (2,0);
						\draw [mark=*] plot (1,1);
						\draw [mark=*, color=red] plot (2,1);
						\end{tikzpicture}} }}
                    child { node {\begin{tikzpicture}[scale=.5]
                        \draw (1,1) node [above] {\tiny $S_1$};
					\draw [mark=*, color=red] plot (1,0);
					\draw [mark=*] plot (0,1);
					\draw [mark=*] plot (0,0);
					\draw [mark=*, color=red] plot (2,0);
					\draw [mark=*, color=red] plot (1,1);
					\draw [mark=*, color=red] plot (2,1);
					\end{tikzpicture}} 
			};
		\end{tikzpicture}
	\end{center}
	\caption{GNSs in $\mathcal{F}(\negc; \negh)$ for $\negc=(3,2)$ and $\negh=(2,1)$ of Example~\ref{exemplo MF}, with the use of the lexicographic order of $\mathbb{N}_0^2$.}
\end{figure}

By employing this idea, the elements of $\mathcal{F}(\negc; \negh)$ can actually be arranged in a rooted subtree of the tree given in \cite[Section 5]{BTT} whose vertices are $\mathcal{F}(\negc)$, with the root also given by the ordinary GNS $\mathcal{O}(\negc)$.

\begin{example} \label{exemplo mf2}
Let $S=\mathcal{O}((3,2))$ be the ordinary GNS of $\mathbb{N}_0^2$ with corner element $\negc=(3,2)$ and consider $\negh=(2,0)$. Let us compute the sets $\mathcal{F}(\negc; \negh)$ and $\mathcal{MF}(\negc; \negh)$ by implementing the modified procedure with the lexicographic order $\preceq$ of $\mathbb{N}_0^2$. Since $\operatorname{CEH}(S)=\{(0,1), (1,1), (2,0), (2,1)\}$, adding the elements of $\operatorname{CEH}(S)$ different from $\negh$, we have the GNSs in $\mathcal{F}(\negc; \negh)$: 
\begin{itemize}
    \item $S_1=S\cup \{(0,1)\}$ with $\operatorname{CEH}(S_1)=\{(1,1), (2, 1)\}$;
    \item $S_2=S\cup \{(1,1)\}$ with $\operatorname{CEH}(S_2)=\{(0,1), (2,0), (2,1)\}$;
    \item $S_3=S\cup \{(2,1)\}$ with $\operatorname{CEH}(S_3)=\{(0,1), (1,1)\}$.
\end{itemize}
For each of these GNSs, we now consider unitary extensions by the corner special gaps different from $\negh$ and smaller than the non-zero elements of GNS ($\leq (2, 1)$) with respect to the lexicographic order $\preceq$ such that :
\begin{itemize}
    \item $S_4=S_2\cup\{(0,1)\}$ with $\operatorname{CEH}(S_4)=\emptyset$;
    \item $S_5=S_3\cup\{(0,1)\}$ with $\operatorname{CEH}(S_5)=\emptyset$;
    \item $S_6=S_3\cup\{(1,1)\}$ with $\operatorname{CEH}(S_6)=\emptyset$.
\end{itemize}
Hence, $\mathcal{F}(\negc; \negh)=\{S,S_1,S_2,\ldots,S_6\}$ and therefore $\mathcal{MF}(\negc; \negh)=\{S_4,S_5, S_6\}$. In particular, the elements of $\mathcal{F}(\negc; \negh)$ can be organized in a tree as in Figure 3.

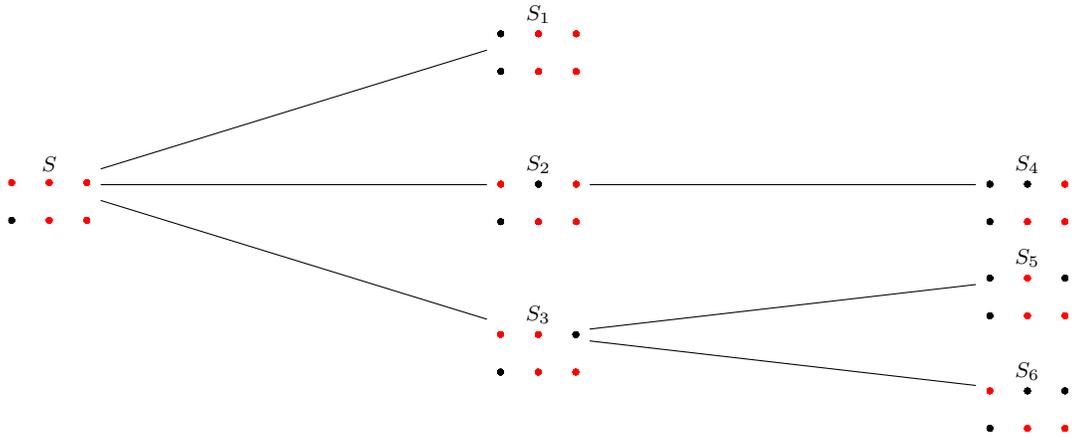
\begin{figure}[htp] \label{fig3}
		\begin{center}
			\begin{tikzpicture}[grow=right, level distance=6.5cm,
  level 1/.style={sibling distance=2cm},
  level 2/.style={sibling distance=1.5cm}]
			\node {\begin{tikzpicture}[scale=.5]
                    \draw (1,1) node [above] {\tiny $S$};
				\draw [mark=*, color=red] plot (1,0);
				\draw [mark=*, color=red] plot (0,1);
				\draw [mark=*] plot (0,0);
				\draw [mark=*, color=red] plot (2,0);
				\draw [mark=*, color=red] plot (1,1);
				\draw [mark=*, color=red] plot (2,1);
				\end{tikzpicture}} 
                child { node {\begin{tikzpicture}[scale=.5]
                        \draw (1,1) node [above] {\tiny $S_3$};
					\draw [mark=*, color=red] plot (1,0);
					\draw [mark=*, color=red] plot (0,1);
					\draw [mark=*] plot (0,0);
					\draw [mark=*] plot (2,1);
					\draw [mark=*, color=red] plot (1,1);
					\draw [mark=*, color=red] plot (2,0);
					\end{tikzpicture}}
				child { node {\begin{tikzpicture}[scale=.5]
                            \draw (1,1) node [above] {\tiny $S_6$};
						\draw [mark=*, color=red] plot (1,0);
						\draw [mark=*, color=red] plot (0,1);
						\draw [mark=*] plot (0,0);
						\draw [mark=*] plot (2,1);
						\draw [mark=*] plot (1,1);
						\draw [mark=*, color=red] plot (2,0);
						\end{tikzpicture}} }
                    child { node {\begin{tikzpicture}[scale=.5]
                            \draw (1,1) node [above] {\tiny $S_5$};
						\draw [mark=*, color=red] plot (1,0);
						\draw [mark=*] plot (0,1);
						\draw [mark=*] plot (0,0);
						\draw [mark=*] plot (2,1);
						\draw [mark=*, color=red] plot (1,1);
						\draw [mark=*, color=red] plot (2,0);
						\end{tikzpicture}} }}
			child { node {\begin{tikzpicture}[scale=.5]
                        \draw (1,1) node [above] {\tiny $S_2$};
					\draw [mark=*, color=red] plot (1,0);
					\draw [mark=*, color=red] plot (0,1);
					\draw [mark=*] plot (0,0);
					\draw [mark=*, color=red] plot (2,0);
					\draw [mark=*] plot (1,1);
					\draw [mark=*, color=red] plot (2,1);
					\end{tikzpicture}} 
				child { node {\begin{tikzpicture}[scale=.5]
                            \draw (1,1) node [above] {\tiny $S_4$};
						\draw [mark=*, color=red] plot (1,0);
						\draw [mark=*] plot (0,1);
						\draw [mark=*] plot (0,0);
						\draw [mark=*, color=red] plot (2,0);
						\draw [mark=*] plot (1,1);
						\draw [mark=*, color=red] plot (2,1);
						\end{tikzpicture}} }}
                    child { node {\begin{tikzpicture}[scale=.5]
                        \draw (1,1) node [above] {\tiny $S_1$};
					\draw [mark=*, color=red] plot (1,0);
					\draw [mark=*] plot (0,1);
					\draw [mark=*] plot (0,0);
					\draw [mark=*, color=red] plot (2,0);
					\draw [mark=*, color=red] plot (1,1);
					\draw [mark=*, color=red] plot (2,1);
					\end{tikzpicture}} 
			};
		\end{tikzpicture}
	\end{center}
	\caption{The tree of GNSs in $\mathcal{F}(\negc; \negh)$ for $\negc=(3,2)$ and $\negh=(2,0)$ of Example~\ref{exemplo mf2}, with respect to the lexicographic order of $\mathbb{N}_0^2$.}
\end{figure}
\end{example}

Varying the elements $\negh$ and computing the sets $\mathcal{MF}(\negc;\negh)$, we obtain the atoms of $\mathcal{F}(\negc)$, possibly with some repetitions. It raises the question about how many elements $\negh$ are necessary to obtain all atoms in $\mathcal{F}(\negc)$ by proceeding in this way. Notice from the Examples \ref{exemplo MF} and \ref{exemplo mf2} that the atoms of $\mathcal{F}((3, 2))$ can be obtained with $\negh\in \{(2,0), (2,1)\}$.

%     \begin{proposition}
% 		Let $S\subseteq \mathbb{N}_0^d$ be a Frobenius GNS. If $S\in \mathcal{MF}(\negc; \negh)$ then $\negc-\mathbf{1}< 2\negh$. 
% 	\end{proposition}
% 	\begin{proof}
%             ....
% 	\end{proof}

% \textcolor{red}{[Conferir o exemplo]}

% \textcolor{red}{
% \begin{remark} We observe that the property of atomicity is not preserved by unitary extensions by adding a $\negc$-special gap. As a matter of fact, let $S$ be the GNS such that
% %
% {\footnotesize $$\operatorname{H}(S) = \{ (0, 1) ,  (0, 2), (0, 3) , (1, 1),  (1, 2),  (1, 3) ,  (2, 1) ,  (2, 2),  (2, 3), (3, 2) ,  (3, 3), (4, 2) ,  (4, 3) , (5, 3)\}.$$}
% %
% In this case, $\operatorname{EH}(S)=\{(4, 2), (5, 3)\}$ and $\operatorname{CEH}(S)=\{(4, 2)\}$. However, \linebreak $S'=S\cup \{(4,2)\}$ is not an atom, since $\operatorname{CEH}(S')=\{(2,1), (3,2)\}$.
% \end{remark}}

% \textcolor{purple}{\dotfill}

\section{On the irreducibility of Frobenius GNSs in terms of a maximal property}
  
In this section, we provide sufficient conditions for certain Frobenius GNSs to be an atom non-irreducible (ANI). Furthermore, we given necessary and sufficient conditions so that maximal elements of a set of Frobenius GNSs with two fixes gaps to be all irreducible or not.  

The following result shows that irreducible GNSs are maximal in the set of the Frobenius GNSs with Frobenius element $\negf$. We point out that there is a characterization of irreducible GNSs with Frobenius element $\negf$ as maximal GNSs that do no contain $\negf$ (see \cite[Proposition 4.11]{CFPU}). In particular, those results extends in a similar way what we have for numerical semigroups (see \cite[Theorem 1]{R-B}).

\begin{proposition}\label{prop irreducible MF}
	Let $S\subseteq \mathbb{N}_0^d$ be a GNS. Then $S$ is irreducible with Frobenius element $\negf$ if and only if $S\in \mathcal{MF}(\negf+\mathbf{1}; \negf)$.
\end{proposition}
\begin{proof} Let $S$ be an irreducible GNS with Frobenius element $\negf$. Hence, we have $|\operatorname{EH}(S)|=1$ by Proposition~\ref{especial irredutivel}, and from Lemma~\ref{maximals especial gap}, we get $\operatorname{EH}(S)=\{\negf\}$. Moreover, since $S$ is a Frobenius GNS, its corner element is $\negf+\mathbf{1}$. It shows that $S\in \mathcal{F}(\negf+\mathbf{1}; \negf)$. Assume now that $S\subsetneq T$ with $T\in \mathcal{F}(\negf+\mathbf{1}; \negf)$. Picking an element $\negx\in Maximals(T\setminus S)$, we get from Lemma~\ref{lemma maximals} that $\negx\in \operatorname{EH}(S)$ with $\negx\neq \negf$ because $\negf\notin T$, which is a contradiction. Therefore, $S\in \mathcal{MF}(\negf+\mathbf{1}; \negf)$.\\
	Conversely, let $S\in \mathcal{MF}(\negf+\mathbf{1}; \negf)$. In particular, $S$ is a Frobenius GNS with $\negf$ its Frobenius element.  If $S$ is not irreducible, then there exist GNSs $S_1$ and $S_2$ such that $S=S_1\cap S_2$ with $S\subsetneq S_1$ and $S\subsetneq S_2$. As $\negf\notin S$, we can suppose without loss of generality that $\negf\notin S_1$. For $\negx\in Maximals(S_1\setminus S)$, we have $\negx\in \operatorname{EH}(S)$ by Lemma~\ref{lemma maximals} and $\negx < \negf$, which yields $\negx\in \operatorname{CEH}(S)$ by Proposition~\ref{c-especial dois elementos}. It implies that $S\cup \{\negx\}\in \mathcal{F}(\negf+\mathbf{1}; \negf)$, which contradicts the maximality of $S$. Therefore, $S$ is irreducible.
\end{proof}

The next result gives sufficient conditions so that every Frobenius GNS  in \linebreak $\mathcal{M}\mathcal{F}(\mathbf{g_2 + 1};\mathbf{g_1})$ to be an atom non-irreducible (ANI).
 
\begin{proposition}\label{prop ANI}
Let $\mathbf{g_1}, \mathbf{g_2}\in\mathbb{N}_0^d$ such that $\mathbf{g_2}$ is a multiple of $\mathbf{g_2}-\mathbf{g_1}$ and $2\mathbf{g_1}\neq \mathbf{g_2}$. Then, every Frobenius GNS in $\mathcal{F}(\mathbf{g_2 + 1};\mathbf{g_1})$ is non-irreducible. In particular, every Frobenius GNS in $\mathcal{M}\mathcal{F}(\mathbf{g_2 + 1};\mathbf{g_1})$ is an ANI.
\end{proposition}
\begin{proof}
Suppose that there is an irreducible Frobenius GNS $S \in \mathcal{F}(\mathbf{g_2 + 1}; \mathbf{g_1})$. In particular, the Frobenius element of $S$ is $\negg_2$, $\negg_1 \notin S$ and we can write $\negg_2 = (\negg_2 - \negg_1) + \negg_1$. By Lemma \ref{Lema soma Frobenius}, it follows that $\mathbf{g_2} - \mathbf{g_1} = \negs \in S$. By hypothesis, there exists $k \in \mathbb{N}_0$ such that $\negg_2 = k(\negg_2 - \negg_1) = k \negs$, but it is a contradiction, since $\negg_2 \notin S$ and $k \negs \in S$. The second part follows from Theorem \ref{atom MF}.
\end{proof}

We can generalize the set $\mathcal{F}(\mathbf{c};\mathbf{h})$ in the following way. Let $\negc, \negh_1,\ldots, \negh_n\in \mathbb{N}_0^d$ with $\negh_i\leq \negc-\mathbf{1}$ for $i=1,\ldots, n$, one can define $\mathcal{F}(\negc; \negh_1,\ldots, \negh_n)$ the set of GNSs $S\subseteq \mathbb{N}_0^d$ with corner element $\negc$ and such that $S\cap \{\negh_1,\ldots,\negh_n\}=\emptyset$. If $\mathcal{MF}(\negc; \negh_1,\ldots, \negh_n)$ stands for the set of maximal elements (with respect to inclusion) of $\mathcal{F}(\negc; \negh_1,\ldots, \negh_n)$, with a similar arguments, the Proposition \ref{CEHunit} can be extended in the following way.

\begin{proposition}\label{CEH e MF}
	Let $S\subseteq \mathbb{N}_0^d$ be a GNS with corner $\negc$ and $\{\negh_1,\ldots, \negh_n\}\subseteq \operatorname{H}(S)$. Then $\operatorname{CEH}(S) \subseteq \{\negh_1,\ldots, \negh_n\}$ if and only if $S\in \mathcal{M} \mathcal{F}(\negc; \negh_1,\ldots, \negh_n)$.
\end{proposition}

Note that $\mathcal{F}(\negg_2+\mathbf{1}; \negg_1)$ contains $\mathcal{F}(\negg_2+\mathbf{1}; \negg_1, \negg_2)$. Furthermore, the Frobenius GNSs in $\mathcal{F}(\negg_2+\mathbf{1}; \negg_1)$ are exactly the elements of $\mathcal{F}(\negg_2+\mathbf{1}; \negg_1, \negg_2)$. Hence, the Frobenius GNSs in $\mathcal{MF}(\negg_2+\mathbf{1}; \negg_1)$ lies in $\mathcal{MF}(\negg_2+\mathbf{1}; \negg_1, \negg_2)$. Now we study the set of Frobenius GNSs with fixed corner and two fixed gaps (one of them is the Frobenius element) which are maximal with respect to the inclusion. The next result gives an important characterization of the elements in $\mathcal{M}\mathcal{F}(\mathbf{g_2 + 1};\mathbf{g_1}, \mathbf{g_2})$ in terms of the set of special gaps.

%\begin{lemma} \label{lemma auxiliar}
%Let $S$ be a GNS with positive genus and corner $\negc$. If $\operatorname{EH}(S)\subseteq \{\mathbf{g_1} < \mathbf{g_2}\}$, then $S$ is Frobenius with $\operatorname{F}(S)=\mathbf{g_2}$.
%\end{lemma}

%\begin{proof}
%Suppose that $\mathbf{g_2} \neq \operatorname{F}(S)$. So, there is $i \in \{1,\ldots,d\}$ such that $\mathbf{g_2} \notin \operatorname{H}(S)^{(i)}$. Logo, $\mathbf{g_2} \notin \operatorname{MH}(S)^{(i)}$. Thus, since $\operatorname{MH}(S)^{(i)} \neq \emptyset$, there is $\mathbf{z} \in \operatorname{MH}(S)^{(i)}$, and by Proposition \ref{maximals especial gap} we have that $\mathbf{z} \in \operatorname{EH}(S)$, a contradiction. Therefore, $\mathbf{g_2} = \operatorname{F}(S)$.
%\end{proof}

\begin{proposition} \label{prop EH MF} 
Let $S\subseteq \mathbb{N}_0^d$ be a GNS and $\negg_1,\negg_2 \in \operatorname{H}(S)$, with $\negg_1 < \negg_2$. Then $S\in \mathcal{MF}(\mathbf{g_2 + 1}; \negg_1,\negg_2)$ if and only if $\operatorname{EH}(S) \subseteq \{\negg_1,\negg_2\}$ .
\end{proposition}
\begin{proof}
$(\Rightarrow)$ Suppose that there exists $\mathbf{x} \in \operatorname{EH}(S)$ such that $\negx \neq \negg_1$ and $\negx \neq \negg_2$. By Proposition \ref{GNS special}, we have that $S \cup \{ \mathbf{x} \}$ is a GNS containing $S$ properly. Since the Frobenius element of $S$ is $\negg_2 + \neg1$, then $\negx < \negg_2$ and the corner element of $S \cup \{\negx\}$ is also $\negg_2 + \neg1$. Hence, $S \cup \{\negx\} \in \mathcal{F}(\negg_2 + \neg1; \negg_1, \negg_2)$, which is a contradiction, by the maximality of $S$.

$(\Leftarrow)$ Suppose that $\operatorname{EH}(S) \subseteq \{\negg_1,\negg_2\}$. So, by Corollary \ref{lemma auxiliar}, \linebreak $S \in \mathcal{F}(\mathbf{g_2 + 1}; \negg_1,\negg_2)$, since $\negg_1 < \negg_2$. Let $T$ be a GNS such that $S \subsetneq T$ and $\mathbf{x}\in Maximals( T\setminus S)$. So, by Lemma \ref{lemma maximals}, $\mathbf{x} \in \operatorname{EH}(S)$. Since  $\operatorname{EH}(S) \subseteq \{\mathbf{g_1},\mathbf{g_2}\}$, we have that $\mathbf{x} = \mathbf{g_i}$ for some $i$, and thus $T \notin \mathcal{F}(\mathbf{g_2 + 1};\mathbf{g_1}, \mathbf{g_2})$. Therefore,  $S\in \mathcal{MF}(\mathbf{g_2 + 1}; \negg_1,\negg_2)$.
\end{proof}

In the follow, we present necessary and sufficient conditions on $\negg_1$ and $\negg_2$ so that every element in $\mathcal{M}\mathcal{F}(\mathbf{g_2 + 1};\mathbf{g_1}, \mathbf{g_2})$ is non-irreducible. By a multiple of an element $\negx \in \mathbb{N}_0^d$, we refer to an element $\negy \in \mathbb{N}_0^d$ such that $\negy = k\negx$, for some $k \in \mathbb{N}$.

\begin{theorem}\label{teo ANI}
Let $\mathbf{g_1}, \mathbf{g_2}\in\mathbb{N}_0^d$. Every element in $\mathcal{M}\mathcal{F}(\mathbf{g_2 + 1};\mathbf{g_1}, \mathbf{g_2})$ is non-irreducible if and only if $\mathbf{g_2}$ is a multiple of $\mathbf{g_2}-\mathbf{g_1}$ and $2\mathbf{g_1}\neq \mathbf{g_2}$. 
\end{theorem}
\begin{proof}
$(\Rightarrow)$ Let $S \in \mathcal{M}\mathcal{F}(\mathbf{g_2 + 1};\mathbf{g_1}, \mathbf{g_2})$. Thus, $\mathbf{g_1}, \mathbf{g_2} \in H(S)$ and $F(S)= \mathbf{g_2} \in EH(S)$. By Propositions \ref{especial irredutivel} and \ref{prop EH MF}, we conclude that $|\operatorname{EH}(S)|=2$, since $S$ is non-irreducible. Suppose that $\operatorname{EH}(S) = \{\mathbf{h}, \mathbf{g_2}\}$ with $\mathbf{h} \neq \mathbf{g_1}$, then $S \subsetneq S \cup \{ \mathbf{h}\}$ and $S \cup \{ \mathbf{h}\} \in \mathcal{F}(\mathbf{g_2 + 1};\mathbf{g_1}, \mathbf{g_2})$, which is a contradiction by the maximality of $S$. Therefore, $\operatorname{EH}(S) = \{\mathbf{g_1} , \mathbf{g_2} \}$, and so $2 \mathbf{g_1} \neq \mathbf{g_2}$. Now, suppose that $\mathbf{g_2}$ is not a multiple of $\mathbf{g_2}-\mathbf{g_1}$ and consider  $\mathcal{O}({\mathbf{g_2}+\mathbf{1}})$ the ordinary GNS with corner $\mathbf{g_2}+\mathbf{1}$. Let $\bar{S}=\langle \mathcal{O}({\mathbf{g_2}+\mathbf{1}}) \cup \{\mathbf{g_2}-\mathbf{g_1}\} \rangle$ be the smallest GNS that contains $\mathcal{O}({\mathbf{g_2}+\mathbf{1}})$ and $\{\mathbf{g_2}-\mathbf{g_1}\}$ with respect to the inclusion. Then $\bar{S}$ is an element of $\mathcal{F}(\mathbf{g_2 + 1}; \mathbf{g_2})$. Take $S \in \mathcal{M}\mathcal{F}(\mathbf{g_2 + 1};\mathbf{g_2})$ such that $\bar{S} \subseteq S$. Note that, by Proposition \ref{prop irreducible MF}, $S$ is irreducible. Since $\mathbf{g_2} \notin S$ and $\mathbf{g_2} - \mathbf{g_1} \in S$, we conclude that $\mathbf{g_1} \notin S$, and so $S \in \mathcal{F}(\mathbf{g_2 + 1}; \mathbf{g_1}, \mathbf{g_2})$. Now, suppose that there is $T \in \mathcal{F}(\mathbf{g_2 + 1};\mathbf{g_1}, \mathbf{g_2})$ such that $S \subsetneq T$. Since $T \in \mathcal{F}(\mathbf{g_2 + 1};\mathbf{g_2})$, we have a contradiction with the maximality of $S$ in the set $\mathcal{F}(\mathbf{g_2 + 1};\mathbf{g_2})$. Thus we get $S \in \mathcal{M}\mathcal{F}(\mathbf{g_2 + 1};\mathbf{g_1}, \mathbf{g_2})$ irreducible, a contradiction. Therefore, $\mathbf{g_2}$ is multiple of $\mathbf{g_2}-\mathbf{g_1}$.

$(\Leftarrow)$ It follows from Proposition \ref{prop ANI}.
\end{proof}

%As a consequence of Theorems \ref{atom MF} and \ref{teo ANI}, every Frobenius GNS in \linebreak $\mathcal{M}\mathcal{F}(\mathbf{g_2 + 1};\mathbf{g_1}, \negg_2)$ is an ANI, since all of them lies in $\mathcal{M}\mathcal{F}(\mathbf{g_2 + 1};\mathbf{g_1})$.

\begin{example}
Let $\mathbf{g_2}=(3,3)$ and $\mathbf{g_1}=(2,2)$. In this case, 
$$\mathcal{MF}(\mathbf{g_2 + 1};\mathbf{g_1},\mathbf{g_2})=\{S_1,\ldots,S_{14}\}, \text{where}$$ 

\noindent$
S_1=\mathbb{N}_0^2\setminus\{ ( 1, 0 ), ( 1, 1 ), ( 2, 0 ), ( 2, 1 ), ( 2, 2 ), ( 3, 0 ), ( 3, 1 ), ( 3, 2 ), ( 3, 3 ) \} \\
S_2=\mathbb{N}_0^2\setminus\{ ( 0, 1 ), ( 1, 0 ), ( 1, 1 ), ( 2, 0 ), ( 2, 1 ), ( 2, 2 ), ( 3, 0 ), ( 3, 1 ), ( 3, 3 )  \} \\
S_3=\mathbb{N}_0^2\setminus\{ ( 0, 1 ), ( 1, 0 ), ( 1, 1 ), ( 1, 2 ), ( 2, 0 ), ( 2, 2 ), ( 3, 0 ), ( 3, 1 ), ( 3, 3 )  \} \\
S_4=\mathbb{N}_0^2\setminus\{ ( 0, 1 ), ( 0, 3 ), ( 1, 0 ), ( 1, 1 ), ( 2, 0 ), ( 2, 1 ), ( 2, 2 ), ( 3, 1 ), ( 3, 3 )  \} \\
S_5=\mathbb{N}_0^2\setminus\{ ( 0, 1 ), ( 0, 3 ), ( 1, 0 ), ( 1, 1 ), ( 1, 2 ), ( 2, 0 ), ( 2, 2 ), ( 3, 1 ), ( 3, 3 )  \} \\
S_6=\mathbb{N}_0^2\setminus\{ ( 0, 1 ), ( 0, 2 ), ( 1, 0 ), ( 1, 1 ), ( 2, 0 ), ( 2, 1 ), ( 2, 2 ), ( 3, 0 ), ( 3, 3 )  \} \\
S_7=\mathbb{N}_0^2\setminus\{( ( 0, 1 ), ( 0, 2 ), ( 1, 0 ), ( 1, 1 ), ( 1, 3 ), ( 2, 1 ), ( 2, 2 ), ( 3, 0 ), ( 3, 3 )  \} \\
S_8=\mathbb{N}_0^2\setminus\{( ( 0, 1 ), ( 0, 2 ), ( 1, 0 ), ( 1, 1 ), ( 1, 2 ), ( 2, 0 ), ( 2, 2 ), ( 3, 0 ), ( 3, 3 ) \} \\
S_9=\mathbb{N}_0^2\setminus\{( ( 0, 1 ), ( 0, 2 ), ( 1, 0 ), ( 1, 1 ), ( 1, 2 ), ( 1, 3 ), ( 2, 2 ), ( 3, 0 ), ( 3, 3 ) \} \\
S_{10}=\mathbb{N}_0^2\setminus\{( ( 0, 1 ), ( 0, 2 ), ( 0, 3 ), ( 1, 1 ), ( 1, 2 ), ( 1, 3 ), ( 2, 2 ), ( 2, 3 ), ( 3, 3 )  \} \\
S_{11}=\mathbb{N}_0^2\setminus\{( ( 0, 1 ), ( 0, 2 ), ( 0, 3 ), ( 1, 0 ), ( 1, 1 ), ( 2, 0 ), ( 2, 1 ), ( 2, 2 ), ( 3, 3 ) \} \\
S_{12}=\mathbb{N}_0^2\setminus\{( ( 0, 1 ), ( 0, 2 ), ( 0, 3 ), ( 1, 0 ), ( 1, 1 ), ( 1, 3 ), ( 2, 1 ), ( 2, 2 ), ( 3, 3 ) \} \\
S_{13}=\mathbb{N}_0^2\setminus\{( ( 0, 1 ), ( 0, 2 ), ( 0, 3 ), ( 1, 0 ), ( 1, 1 ), ( 1, 2 ), ( 2, 0 ), ( 2, 2 ), ( 3, 3 )  \} \\
S_{14}=\mathbb{N}_0^2\setminus\{( ( 0, 1 ), ( 0, 2 ), ( 0, 3 ), ( 1, 0 ), ( 1, 1 ), ( 1, 2 ), ( 1, 3 ), ( 2, 2 ), ( 3, 3 ) \}
$

By Theorem \ref{teo ANI}, we conclude that $S_i$ is non-irreducible, for all $i = 1,\ldots,14$.
\end{example}

 For a GNS $S \subseteq \mathbb{N}_0^d$ and $\mathbf{g_1}, \mathbf{g_2} \in \mathbb{N}_0^d$, define $$\tilde{\mathcal{M}}_{\mathbf{g_1}, \mathbf{g_2}}(S):=Maximals\{\mathbf{x}\in \operatorname{H}(S):2\mathbf{x}\in S, \mathbf{g_i}-\mathbf{x}\notin S, \mbox{ for }i=1,2\}.$$

\begin{lemma} \label{lemma h maximals}
Let $\mathbf{g_1},\mathbf{g_2}\in \mathbb{N}_0^d$ and $S\in \mathcal{F}(\mathbf{g_2 + 1};\mathbf{g_1}, \mathbf{g_2})$. If $\displaystyle \mathbf{h}\in \tilde{\mathcal{M}}_{\mathbf{g_1}, \mathbf{g_2}}(S)$,
then $S\cup\{\mathbf{h}\}\in \mathcal{F}(\mathbf{g_2 + 1};\mathbf{g_1}, \mathbf{g_2})$. Moreover, $S\in \mathcal{M}\mathcal{F}(\mathbf{g_2 + 1};\mathbf{g_1}, \mathbf{g_2})$ if and only if $\displaystyle  \tilde{\mathcal{M}}_{\mathbf{g_1}, \mathbf{g_2}}(S)= \emptyset$.
\end{lemma}
\begin{proof}
Note that, if there exist an element $\mathbf{h}\in \tilde{\mathcal{M}}_{\mathbf{g_1}, \mathbf{g_2}}(S)$, then $\mathbf{h}\neq \mathbf{g_i}$ for any $i\in\{1,2\}$ and, $2\mathbf{h}\in S$. Let $\mathbf{s}\in S\setminus\{0\}$ and suppose that $\mathbf{h}+\mathbf{s}\not\in S$. So, we have $\mathbf{h}+\mathbf{s}\in \operatorname{H}(S)$ and $\mathbf{h}<\mathbf{h}+\mathbf{s}$. Its clear that $2(\mathbf{h}+\mathbf{s})\in S$. If $\mathbf{g_i}-(\mathbf{h}+\mathbf{s})\in S$, then $\mathbf{g_i}-\mathbf{h}\in S$, contradiction. Thus $\mathbf{g_i}-(\mathbf{h}+\mathbf{s})\not\in S$, and then it follows that $\mathbf{h}+\mathbf{s}\in  \tilde{\mathcal{M}}_{\mathbf{g_1}, \mathbf{g_2}}(S)$, a contradiction with the maximality of $\mathbf{h}$. Therefore, $\mathbf{h}\in \operatorname{EH}(S)$ and we have $S\cup\{\mathbf{h}\}\in \mathcal{F}(\mathbf{g_2 + 1}; \mathbf{g_1},\mathbf{g_2})$. 

To prove the last part of the lemma, first let $S\in \mathcal{M}\mathcal{F}(\mathbf{g_2 + 1};\mathbf{g_1}, \mathbf{g_2})$ and suppose that there is $\mathbf{h}\in  \tilde{\mathcal{M}}_{\mathbf{g_1}, \mathbf{g_2}}(S)$. Therefore, $S \subsetneq S\cup\{\mathbf{h}\}$, and we have a contradiction, since $S\cup\{\mathbf{h}\} \in \mathcal{F}(\mathbf{g_2 + 1}; \mathbf{g_1},\mathbf{g_2})$.

Now, let $ \tilde{\mathcal{M}}_{\mathbf{g_1}, \mathbf{g_2}}(S) = \emptyset$. Suppose that $S \subsetneq S'$, where $S'$ is a GNS. Let $\mathbf{x}\in S'\setminus S$ be such that $2\mathbf{x}\in S$. Let $\langle S,\mathbf{x} \rangle$ be the smallest GNS that contain $S$ and $\{ \negx \}$ with respect to inclusion. Then $S\subsetneq \langle S,\mathbf{x} \rangle \subseteq S'$ and, by hypothesis, there exist $i\in\{1,2\}$ such that $\mathbf{g_i}-\mathbf{x}\in S$. So, $\mathbf{g_i}\in \langle S,\mathbf{x} \rangle$, which leads to $\mathbf{g_i}\in S'$, that is, $S' \notin \mathcal{F}(\mathbf{g_2 + 1}; \mathbf{g_1}, \mathbf{g_2})$. Therefore, we have that $S\in \mathcal{M}\mathcal{F}(\mathbf{g_2 + 1};\mathbf{g_1}, \mathbf{g_2})$, and the result follows.
\end{proof}

For a GNS $S \subseteq \mathbb{N}_0^d$ and $\negg \in \mathbb{N}_0^d$, define $$\mathcal{M}_{\negg}(S):=Maximals\{\mathbf{x}\in \operatorname{H}(S):2\mathbf{x}\neq \mathbf{g}, \mathbf{g}-\mathbf{x}\notin S\}.$$

\begin{remark} \label{remark 2h}
We observe that, if $\mathbf{h} \in \mathcal{M}_{\negg}(S)$, then $2\mathbf{h}>\mathbf{g}$. In fact, note that $\mathbf{h} \in \mathcal{M}_{\negg}(S)$ implies $\mathbf{g} - \negh \in  \mathcal{M}_{\negg}(S)$. Thus, if $2\mathbf{h}<\mathbf{g}$, then $\negh < \mathbf{g} - \negh \in  \mathcal{M}_{\negg}(S)$, a contradiction with the maximality of $\negh$.
\end{remark}

In the following, we give some results on GNSs $S$ such that $EH(S) = \{\mathbf{g_1} < \mathbf{g_2}\}$.
%\begin{lemma} \label{lemma auxiliar}
%Let $S$ be a GNS and $\mathbf{g} \in H(S)$. Then $\mathbf{g} \in EH(S)$ if and only if $\not\exists \mathbf{g}' \in H(S)$ such that $\mathbf{g}' > \mathbf{g}$.
%\end{lemma}

\begin{lemma}
Let $S$ be a GNS such that $\operatorname{EH}(S)=\{\mathbf{g_1}<\mathbf{g_2}\}$. Then $$\mathcal{M}_{\negg_2}(S)= \{\mathbf{g_1}\}.$$
\end{lemma}
\begin{proof}
Since $|\operatorname{EH}(S)|=2$, Proposition \ref{especial irredutivel}, we have that $S$ is not irreducible, and we can conclude that the set $\mathcal{M}_{\negg_2}(S)$ is not empty. Let $\negh \in \mathcal{M}_{\negg_2}(S)$. By Corollary \ref{lemma auxiliar}, $\mathbf{g_2} = F(S)$. So, by Remark \ref{remark 2h}, we have $2 \negh \in S$. Suppose that there is $\negs \in S \setminus \{0\}$ such that $\negh + \negs \notin S$. Note that, $2(\negh + \negs) \neq \mathbf{g_2}$. Thus, by maximality of $\negh$, we have that $\mathbf{g_2} - (\negh + \negs) \in S$, and so $\mathbf{g_2}-\mathbf{h}\in S$, a contradiction. Therefore, $\mathbf{h} \in \operatorname{EH}(S)$, and the result follows, since $\operatorname{EH}(S)=\{\mathbf{g_1}<\mathbf{g_2}\}$ and $\mathbf{g_2} \notin \mathcal{M}_{\negg_2}(S)$.
\end{proof}

So, by Remark \ref{remark 2h}, we obtain the following consequence.

\begin{lemma} \label{lemma 10 rosales}
Let $S$ be a GNS such that $\operatorname{EH}(S)=\{\mathbf{g_1}<\mathbf{g_2}\}$. Then $2\mathbf{g_1}>\mathbf{g_2}$.
\end{lemma}

In the follow, we provide the conditions on $\mathbf{g_1}$ and $\mathbf{g_2}$ for the elements in $\mathcal{MF}(\mathbf{g_2 + 1};\mathbf{g_1},\mathbf{g_2})$ have the set of special gaps equal to $\{\mathbf{g_1},\mathbf{g_2}\}$. 

\begin{proposition} \label{prop 12 rosales}
Let $\mathbf{g_1},\mathbf{g_2}\in\mathbb{N}_0^d$ be such that $\mathbf{g_2}<2\mathbf{g_1}<2\mathbf{g_2}$. The following statements are equivalent.
\begin{enumerate}
	\item $S$ is a GNS with $\operatorname{EH}(S)=\{\mathbf{g_1},\mathbf{g_2}\}$.
	\item $S\in\mathcal{MF}(\mathbf{g_2 + 1};\mathbf{g_1},\mathbf{g_2})$ and $\mathbf{g_2}-\mathbf{g_1}\not\in S$.
\end{enumerate}
\end{proposition}

\begin{proof}
$(1) \Rightarrow (2)$ By Proposition \ref{prop EH MF}, $S\in\mathcal{MF}(\mathbf{g_2 + 1};\mathbf{g_1},\mathbf{g_2})$ . If $\mathbf{g_2} - \mathbf{g_1} \in S$, then we get $\mathbf{g_2} = \mathbf{g_1} +(\mathbf{g_2} - \mathbf{g_1}) \in S$, since $\mathbf{g_1} \in \operatorname{EH}(S)$. Therefore, $\mathbf{g_2} - \mathbf{g_1} \notin S$.

$(2) \Rightarrow (1)$  By Proposition \ref{prop EH MF}, $\operatorname{EH}(S)\subseteq \{\mathbf{g_1},\mathbf{g_2}\}$. By Corollary \ref{lemma auxiliar}, we have that $\mathbf{g_2} = F(S) \in \operatorname{EH}(S)$. We claim that $\mathbf{g_1}$ is also in $\operatorname{EH}(S)$. In fact, as $2 \mathbf{g_1} > \mathbf{g_2}$, then $2 \mathbf{g_1} \in S$. If there is $\negs \in S \setminus \{ 0 \}$ such that $\mathbf{g_1} + \negs \notin S$, then, by Lemma \ref{lemma h maximals}, we have $\mathbf{g_2} - (\mathbf{g_1} + \negs) \in S$, and so $\mathbf{g_2} - \mathbf{g_1} \in S$, a contradiction. Therefore, $\operatorname{EH}(S) = \{ \mathbf{g_1} , \mathbf{g_2}\}$.
\end{proof}

Before to introduce the next results, we observe that if $S$ is a GNS and $\mathbf{t} \in \operatorname{PF}(S)$, then there is a positive integer $k$ such that $k  \mathbf{t} \in \operatorname{EH}(S)$.

\begin{lemma}\label{lemma type}
Let $S$ be a GNS and $\negh \in \operatorname{H}(S) \setminus \operatorname{PF}(S)$. Then $\{\negx \in \operatorname{EH}(S): \negx - \negh \in S\} \neq \emptyset$.
\label{nonempty}
\end{lemma}

\begin{proof}
Observe that, if $\negh\in \operatorname{H}(S)\setminus \operatorname{PF}(S)$, then there exists $\negs\in S \setminus \{0\}$ such that $\negh+\negs\in \operatorname{H}(S)$. Let $\negs_1 \in S \setminus \{0\}$ be such that $\negh+\negs_1\in \operatorname{H}(S)$ and suppose that there is none $\negs \in S \setminus \{0\}$ such that $\negs_1< \negs$ and $\negh+\negs\in \operatorname{H}(S)$. We claim that $\negx=\negh+\negs_1 \in \operatorname{EH}(S)$. In fact, by the choice of $\negs_1$, we have that $\negx+\negs \in S$, for all $\negs \in S \setminus \{0\}$. Note that, $\negh+2\negs_1\in S$. So, $2\negx = \negh+(\negh+2\negs_1)\in S$ and the result follows.
\end{proof}

\begin{lemma} \label{lemma 17 rosales}
Let $S$ be a GNS with $\operatorname{EH}(S) = \{ \mathbf{g_1} < \mathbf{g_2} \}$. Then either $\mathbf{g_2} = k  (\mathbf{g_2} - \mathbf{g_1})$, for some positive integer $k$, or $2 \mathbf{g_1} - \mathbf{g_2} \in S$.
\end{lemma}

\begin{proof}
From Proposition \ref{prop 12 rosales} it follows that $\mathbf{g_2} - \mathbf{g_1} \in \operatorname{H}(S)$. Now, we distinguish two cases.

$\bullet$ If $\mathbf{g_2} - \mathbf{g_1} \notin \operatorname{PF}(S)$, then Lemma \ref{lemma type} asserts that either $\mathbf{g_2} - (\mathbf{g_2} - \mathbf{g_1}) \in S$ or $\mathbf{g_1} - (\mathbf{g_2} - \mathbf{g_1}) \in S$. The first condition cannot hold, and thus $2 \mathbf{g_1} - \mathbf{g_2} \in S$.

$\bullet$ If $\mathbf{g_2} - \mathbf{g_1} \in \operatorname{PF}(S)$, then there is a positive integer $k$ such that $k (\mathbf{g_2} - \mathbf{g_1}) \in \operatorname{EH}(S)$. Hence either $\mathbf{g_2} = \alpha  (\mathbf{g_2} - \mathbf{g_1})$ or $\mathbf{g_1} = \beta  (\mathbf{g_2} - \mathbf{g_1})$, for some $\alpha,\beta\in\mathbb{N}$. The proofs follows from the fact that if $\mathbf{g_1} = \beta (\mathbf{g_2} - \mathbf{g_1})$, then $\mathbf{g_2} = (\beta + 1) (\mathbf{g_2} - \mathbf{g_1})$.
\end{proof}

%\begin{proposition}
%Let $S$ be a GNS and $\mathbf{g_1}, \mathbf{g_2}\in\mathbb{N}_0^d$ such that $\mathbf{g_2}< 2\mathbf{g_1}<2\mathbf{g_2}$ and $\mathbf{g_2}$ is not multiply of $\mathbf{g_2}-\mathbf{g_1}$. Then $\operatorname{EH}(S)=\{\mathbf{g_1},\mathbf{g_2}\}$ if and only if $S\in\mathcal{M}L(\mathbf{g_1},\mathbf{g_2})$ and $2\mathbf{g_1}-\mathbf{g_2}\in S$.
%\end{proposition}

%\begin{proof}
%$(\Rightarrow)$ Follows directly from Proposition \ref{prop 12 rosales} and Lemma \ref{lemma 17 rosales}. 

%$(\Leftarrow)$ From Proposition \ref{prop 12 rosales} it is enough to prove that $\mathbf{g_2}-\mathbf{g_1}\not\in S$. Note that, since $\mathbf{g_1}=(2\mathbf{g_1}-\mathbf{g_2})+(\mathbf{g_2}-\mathbf{g_1})$, if $\mathbf{g_2}-\mathbf{g_1}\in S$, then $\mathbf{g_1}\in S$, which is impossible. So, $\mathbf{g_2}-\mathbf{g_1}\not\in S$ and the result follows.
%\end{proof}

%The reader can easy prove the follow result.

%\begin{lemma}
%Let $\mathbf{g_2}, \mathbf{g_1}\in\mathbb{N}^d$ such that $\mathbf{g_1}=k (2\mathbf{g_1} - \mathbf{g_2})$. Then $\mathbf{g_2}=k'(2\mathbf{g_1} - \mathbf{g_2})$.
%\end{lemma}

In the next result we have a characterization of GNSs with exactly two comparable special gaps. As consequence we get the necessary and sufficient conditions so that every elements in $\mathcal{MF}(\mathbf{g_2 + 1};\mathbf{g_1},\mathbf{g_2})$ to be irreducible.

\begin{theorem}\label{TmainS5}
Let $\mathbf{g_1}, \mathbf{g_2}\in\mathbb{N}_0^d$ such that $\mathbf{g_1}< \mathbf{g_2}$. Then there exist $S$ a GNS such that $\operatorname{EH}(S)=\{\mathbf{g_1},\mathbf{g_2}\}$ if and only if $\mathbf{g_2}<2\mathbf{g_1}$ and either $\mathbf{g_2}=k (\mathbf{g_2}-\mathbf{g_1})$ or $\mathbf{g_2}$ is not multiple of $2\mathbf{g_1}-\mathbf{g_2}$.
\end{theorem}

\begin{proof}
$(\Rightarrow)$ Let $S$ be a GNS with  $\operatorname{EH}(S)=\{\mathbf{g_1},\mathbf{g_2}\}$. By Lemma \ref{lemma 10 rosales}, we know that $2 \mathbf{g_1} > \mathbf{g_2}$ and by Lemma \ref{lemma 17 rosales} that either  $\mathbf{g_2} = k \cdot (\mathbf{g_2} - \mathbf{g_1})$, for some positive integer $k$, or $2 \mathbf{g_1} - \mathbf{g_2} \in S$. If this latter condition holds, then $\mathbf{g_2}$ is not multiple of $2\mathbf{g_1}-\mathbf{g_2}$, since $\mathbf{g_2} \notin S$.

$(\Leftarrow)$ First, assume that $\mathbf{g_2}=k\cdot (\mathbf{g_2}-\mathbf{g_1})$ and let $S \in \mathcal{MF}(\mathbf{g_2 + 1}; \mathbf{g_1}, \mathbf{g_2})$. By Lemma \ref{prop EH MF} we know that $\operatorname{EH}(S) \subseteq \{ \mathbf{g_1} , \mathbf{g_2} \}$. Note that, $\mathbf{g_2} - \mathbf{g_1} \notin S$. So by Proposition \ref{prop 12 rosales} we have $\operatorname{EH}(S)=\{\mathbf{g_1},\mathbf{g_2}\}$. Now, assume that $\mathbf{g_2}$ is not multiple of $2\mathbf{g_1}-\mathbf{g_2}$. Since $\mathbf{g_1} = k (2\mathbf{g_1} - \mathbf{g_2})$ implies $\mathbf{g_2} = (2k-1) (2\mathbf{g_1} - \mathbf{g_2})$, we have that $\mathbf{g_1}$ also is not multiple of $2\mathbf{g_1} - \mathbf{g_2}$. Let $\bar{S}=\langle \mathcal{O}({\mathbf{g_2}+\mathbf{1}}) \cup \{2\mathbf{g_1}-\mathbf{g_2}\} \rangle$ be the smallest GNS that contain $\mathcal{O}({\mathbf{g_2}+\mathbf{1}})$ and $\{2\mathbf{g_1}-\mathbf{g_2}\}$ with respect to inclusion. So, $\bar{S} \in \mathcal{F}(\mathbf{g_2 + 1}; \mathbf{g_1}, \mathbf{g_2})$. Take $S \in \mathcal{MF}(\mathbf{g_2 + 1}; \mathbf{g_1}, \mathbf{g_2})$ be such that $\bar{S} \subseteq S$. Thus, by Proposition \ref{prop 12 rosales}, it suffices to show that $\mathbf{g_2} - \mathbf{g_1} \notin S$. Suppose that, $\mathbf{g_2} - \mathbf{g_1} \in S$. Thus, $\mathbf{g_1} = (2\mathbf{g_1} - \mathbf{g_2}) + (\mathbf{g_2} - \mathbf{g_1}) \in S$, a contradiction. 
\end{proof}

\begin{corollary}\label{Cor GNS Irreducible}
Let $\mathbf{g_1}, \mathbf{g_2}\in\mathbb{N}_0^d$ such that $\mathbf{g_1}< \mathbf{g_2}$. Then all the elements of $\mathcal{MF}(\mathbf{g_2 + 1};\mathbf{g_1},\mathbf{g_2})$ are irreducible GNS if and only if one of the following conditions holds:
\begin{enumerate}
	\item $2\mathbf{g_1}\leq \mathbf{g_2}\;,$
	\item $\mathbf{g_2}$ is not multiple of $\mathbf{g_2}-\mathbf{g_1}$ and $\mathbf{g_2}$ is multiple of $2\mathbf{g_1}-\mathbf{g_2}$.
\end{enumerate}
\end{corollary}

\begin{example}
Let $\mathbf{g_2}=(3,3)$ and $\mathbf{g_1}=(1,1)$. In this case, 
$$\mathcal{MF}(\mathbf{g_2 + 1};\mathbf{g_1},\mathbf{g_2})=\{S_1,\ldots,S_{22}\}, \text{where}$$ 

\noindent$
S_1=\mathbb{N}^2_0\setminus\{ ( 1, 0 ), ( 1, 1 ), (2, 0 ), (2, 1 ), ( 3, 0 ), (3, 1 ), ( 3, 2 ), ( 3, 3 ) \} \\
S_2=\mathbb{N}^2_0\setminus\{ ( 1, 0 ), ( 1, 1 ), ( 1, 2 ), ( 2, 0 ), ( 3, 0 ), ( 3, 1 ), ( 3, 2 ), ( 3, 3 ) \} \\
S_3=\mathbb{N}^2_0\setminus\{ ( 1, 0 ), ( 1, 1 ), ( 1, 2 ), ( 1, 3 ), ( 3, 0 ), ( 3, 1 ), ( 3, 2 ), ( 3, 3 ) \} \\
S_4=\mathbb{N}^2_0\setminus\{ ( 0, 1 ), ( 1, 0 ), ( 1, 1 ), ( 2, 0 ), ( 2, 1 ), ( 3, 0 ), ( 3, 1 ), ( 3, 3 ) \} \\
S_5=\mathbb{N}^2_0\setminus\{ ( 0, 1 ), ( 1, 0 ), ( 1, 1 ), ( 1, 3 ), ( 2, 1 ), ( 3, 0 ), ( 3, 1 ), ( 3, 3 ) \} \\
S_6=\mathbb{N}^2_0\setminus\{ ( 0, 1 ), ( 1, 0 ), ( 1, 1 ), ( 1, 2 ), ( 2, 0 ), ( 3, 0 ), ( 3, 1 ), ( 3, 3 ) \} \\
S_7=\mathbb{N}^2_0\setminus\{ ( 0, 1 ), ( 1, 0 ), ( 1, 1 ), ( 1, 2 ), ( 1, 3 ), ( 3, 0 ), ( 3, 1 ), ( 3, 3 ) \} \\
S_8=\mathbb{N}^2_0\setminus\{ ( 0, 1 ), ( 0, 3 ), ( 1, 1 ), ( 1, 3 ), ( 2, 1 ), ( 2, 3 ), ( 3, 1 ), ( 3, 3 ) \} \\
S_9=\mathbb{N}^2_0\setminus\{ ( 0, 1 ), ( 0, 3 ), ( 1, 0 ), ( 1, 1 ), ( 2, 0 ), ( 2, 1 ), ( 3, 1 ), ( 3, 3 ) \} \\
S_{10}=\mathbb{N}^2_0\setminus\{ ( 0, 1 ), ( 0, 3 ), ( 1, 0 ), ( 1, 1 ), ( 1, 3 ), ( 2, 1 ), ( 3, 1 ), ( 3, 3 ) \} \\
S_{11}=\mathbb{N}^2_0\setminus\{  ( 0, 1 ), ( 0, 3 ), ( 1, 0 ), ( 1, 1 ), ( 1, 2 ), ( 2, 0 ), ( 3, 1 ), ( 3, 3 ) \} \\
S_{12}=\mathbb{N}^2_0\setminus\{ ( 0, 1 ), ( 0, 3 ), ( 1, 0 ), ( 1, 1 ), ( 1, 2 ), ( 1, 3 ), ( 3, 1 ), ( 3, 3 ) \} \\
 S_{13}=\mathbb{N}^2_0\setminus\{ ( 0, 1 ), ( 0, 2 ), ( 1, 0 ), ( 1, 1 ), ( 2, 0 ), ( 2, 1 ), ( 3, 0 ), ( 3, 3 ) \} \\
S_{14}=\mathbb{N}^2_0\setminus\{  ( 0, 1 ), ( 0, 2 ), ( 1, 0 ), ( 1, 1 ), ( 1, 3 ), ( 2, 1 ), ( 3, 0 ), ( 3, 3 ) \} \\
S_{15}=\mathbb{N}^2_0\setminus\{ ( 0, 1 ), ( 0, 2 ), ( 1, 0 ), ( 1, 1 ), ( 1, 2 ), ( 2, 0 ), ( 3, 0 ), ( 3, 3 ) \} \\
S_{16}=\mathbb{N}^2_0\setminus\{ ( 0, 1 ), ( 0, 2 ), ( 1, 0 ), ( 1, 1 ), ( 1, 2 ), ( 1, 3 ), ( 3, 0 ), ( 3, 3 ) \} \\
S_{17}=\mathbb{N}^2_0\setminus\{  ( 0, 1 ), ( 0, 2 ), ( 0, 3 ), ( 1, 1 ), ( 1, 3 ), ( 2, 1 ), ( 2, 3 ), ( 3, 3 ) \} \\
S_{18}=\mathbb{N}^2_0\setminus\{  ( 0, 1 ), ( 0, 2 ), ( 0, 3 ), ( 1, 1 ), ( 1, 2 ), ( 1, 3 ), ( 2, 3 ), ( 3, 3 ) \} \\
S_{19}=\mathbb{N}^2_0\setminus\{  ( 0, 1 ), ( 0, 2 ), ( 0, 3 ), ( 1, 0 ), ( 1, 1 ), ( 2, 0 ), ( 2, 1 ), ( 3, 3 ) \} \\
S_{20}=\mathbb{N}^2_0\setminus\{  ( 0, 1 ), ( 0, 2 ), ( 0, 3 ), ( 1, 0 ), ( 1, 1 ), ( 1, 3 ), ( 2, 1 ), ( 3, 3 ) \} \\
S_{21}=\mathbb{N}^2_0\setminus\{  ( 0, 1 ), ( 0, 2 ), ( 0, 3 ), ( 1, 0 ), ( 1, 1 ), ( 1, 2 ), ( 2, 0 ), ( 3, 3 ) \} \\
 S_{22}=\mathbb{N}^2_0\setminus\{ ( 0, 1 ), ( 0, 2 ), ( 0, 3 ), ( 1, 0 ), ( 1, 1 ), ( 1, 2 ), ( 1, 3 ), ( 3, 3 ) \}
 $
 
By Corollary \ref{Cor GNS Irreducible}, $S_i$ is irreducible for all $i = 1,\ldots,22$.
\end{example}

	\end{document}